\documentclass[12pt,a4paper,reqno]{amsart}
\usepackage[utf8]{inputenc}
\usepackage[T1]{fontenc}
\usepackage[foot]{amsaddr}
\usepackage{amsmath}
\usepackage{amsfonts}
\usepackage{amssymb}
\usepackage{algorithm}
\usepackage{algpseudocode}
\usepackage{xcolor}
\usepackage{soul}
\usepackage{graphicx}
\usepackage{wrapfig}
\usepackage{subcaption}
\usepackage[nopar]{lipsum}
\usepackage{amsthm}
\usepackage{amssymb}
\usepackage{euscript}
\usepackage{blkarray}
\usepackage{makecell, multirow, tabularx,booktabs}
\usepackage{enumitem}
\usepackage{mathtools}
\usepackage{soul}
\usepackage{amscd}
\usepackage{float}
\usepackage[pagebackref,colorlinks=true,
linkcolor=blue,
urlcolor=red,colorlinks,
citecolor=green]{hyperref}
\usepackage{cleveref}
\usepackage{geometry}
\usepackage{pgf,tikz,pgfplots}
\tikzstyle{none}=[]
\tikzstyle{new style 0}=[draw,circle,fill=white]
\tikzstyle{new edge style 1}=[draw,dashed]
\tikzstyle{new edge style 1}=[draw,dashed]
\usepackage{adjustbox}
\usepackage{setspace}
\usepackage[numbers,sort&compress]{natbib}

\allowdisplaybreaks[4]
\usetikzlibrary{calc}
\usetikzlibrary{patterns}
\pgfplotsset{compat=1.14}
\numberwithin{equation}{section}
\newtheorem{thm}{Theorem}[section]
\newtheorem{prop}[thm]{Proposition}

\newtheorem{cor}[thm]{Corollary}
\newtheorem{exm}[thm]{Example}
\newtheorem{df}[thm]{Definition}

\newtheorem{fac}[thm]{Fact}

\makeatletter
\@namedef{subjclassname@1991}{Subject}
\@namedef{subjclassname@2000}{Subject}
\@namedef{subjclassname@2010}{Subject}
\@namedef{subjclassname@2020}{Subject}
\makeatother
\begin{document}
	\title{On the Incidence matrices of hypergraphs}
	\author[ Parui]{Samiron Parui} 
	\email[ Parui ]{\textit {{\scriptsize  samironparui@gmail.com}}}
	
	\address{School of Mathematical Sciences,  National Institute of Science Education and Research Bhubaneswar,  Bhubaneswar, Padanpur, Odisha 752050, India}

	
	\date{\today}
	\keywords{ Hypergraphs; Incidence matrix; Graphs and linear algebra; Units in hypergraphs; 
 Adjacency Matrix 
 }
	
	\subjclass[2020]{Primary 05C65; 
 15A03
		; 05C50; 
		 Secondary 
		 	05C07;
		 	05C30
	}
	
	\maketitle
	\begin{abstract}
	    This study delves into the incidence matrices of hypergraphs, with a focus on two types: the edge-vertex incidence matrix and the vertex-edge incidence matrix. The edge-vertex incidence matrix is a matrix in which the rows represent hyperedges and the columns represent vertices. For a given hyperedge $e$ and vertex $u$, the $(e,u)$-th entry of the matrix is $1$ if $u$ is incident to $e$; otherwise, this entry is $0$. The vertex-edge incidence matrix is simply the transpose of the edge-vertex incidence matrix. This study examines the ranks and null spaces of these incidence matrices. It is shown that certain hypergraph structures, such as $k$-uniform cycles, units, and equal partitions of hyperedges and vertices, can influence specific vectors in the null space. In a hypergraph, a unit is a maximal collection of vertices that are incident with the same set of hyperedges. Identification of vertices within the same unit leads to a smaller hypergraph, known as unit contraction. The rank of the edge-vertex incidence matrix remains the same for both the original hypergraph and its unit contraction. Additionally, this study establishes connections between the edge-vertex incidence matrix and certain eigenvalues of the adjacency matrix of the hypergraph.

	\end{abstract}
\section{Introduction}
The interrelation between the graph structure and the incidence matrix of the graph is well-studied in the literature \cite{nufflen-incidence-1,Brualdi-comb-mat,akbari-incidence,haemer-singular-incidence}. The rank of the incidence matrix is affected by the structure of the graph. For instance, a connected graph on $n$ vertices with $n$ edges contains a unique cycle $C$. The rank of its edge-vertex incidence matrix is $n$ or $n-1$, respectively, if the length of $C$ is odd or even (see \cite[P.37, Exercise-7]{Brualdi-comb-mat}). In a bipartite graph $G$ with the bipartition of the vertex set $V(G)=V_1\cup V_2$, the edge-vertex incidence matrix can never be of full rank. The vector $x:V(G)\to\{1,-1\}$ with $x(v)=1$ if $v\in V_1$ and $x(v)=-1$ for all $v\in V_2$ always belongs to the null space of the edge-vertex incidence matrix. In a graph on $n$ vertices and $m$ edges with $p$ number of bipartite components and $q$ isolated vertices, the rank of the vertex-edge incidence matrix associated with the graph is $n-p-q$ \cite{nufflen-incidence-1}. Here, we explore similar properties in hypergraphs. We consider a hypergraph, named $k$-uniform cycle of length $n$, such that it coincides with the cycle graph $C_n$ for the $k=2$ case. For the graph case (that is, the $k=2$ case), if the length $n$ is even, then using $-1$ and $1$ alternatively, we can have a vector $x$ in the null space of the edge-vertex incidence matrix of $C_n$. Here $1$ and $-1$ are the $2$-nd roots of the unity. Similarly, for hypergraphs, we show here that the $k$-th roots of unity can be used to describe the vectors in the null spaces of the edge-vertex incidence matrix of $k$ uniform cycles (see \Cref{x-w-prop}, \Cref{prop-cycle}, \Cref{gcd}).

In the context of a graph, the ranks of the incidence matrices of the graph are affected by the existence of specific substructures like even cycles, bipartite components, etc. (\cite{Brualdi-comb-mat,nufflen-incidence-1}). Similarly, we observe that the existence of specific substructures in a hypergraph is reflected in the null space of its incidence matrix (see the \Cref{thm-induced} and the \Cref{cycle-induced}). In any bipartite graph $G$ with the bipartition of the vertex set $V(G)=V_1\cup V_2$, for any edge $e$ in $G$, we have $|e\cap V_1|=|e\cap V_2|=1$. This property leads us to the vector $x:V(G)\to\{1,-1\}$ in the null space of the edge-vertex incidence matrix of the bipartite graph $G$ such that $x(v)=1$ if $v\in V_1$ and $x(v)=-1$ for all $v\in V_2$. We extend this property for hypergraphs and named it \emph{equal partition of hyperedges} (see \Cref{equal partition defn}). Given a hypergraph $H$ with the vertex set $V(H)$, a pair of disjoint subsets $V_1,V_2\subset V(H)$ is called an equal partition of hyperedges if $|e\cap V_1|=|e\cap V_2|$ for all hyperedges $e$ in $H$. The interrelation of this structure with the rank and null space of the edge-vertex incidence matrix of hypergraphs is described in the \Cref{thm-equal-partition}. In $\mathbb{R}^2$, the point $(0,0)$ is the midpoint of the line segment connecting $(-1,-1)$ and $(1,1)$. An equal partition of hyperedges and the vector in the null space of the edge-vertex incidence matrix due to the equal partition is similar to this fact from coordinate geometry. This fact has a generalization. If $(-a_1,-a_2)$ and $(b_1,b_2)$ are two endpoints of a line segment such that the origin $(0,0)$ divides the line segment in a ratio $p:q$, then $a_i:b_i=p:q$ for all $i=1,2$. This fact motivates the question of the existence of a substructure in a hypergraph that is a generalization of the equal partition of hyperedges and corresponds to vectors in the null space of the incidence matrix of the hypergraph. Positive answers to this question are presented as the \cref{ratio-r}, \Cref{uvw}, and \Cref{gen-ker}.  If $V_1$, and $V_2$ are two disjoint subsets of the vertex set with $|V_1\cap e|:|V_2\cap e|=r$ for all hyperedge $e$ in the hypergraph $H$, then the pair of sets $V_1$, and $V_2$ corresponds to a vector in the null space of the edge-vertex incidence matrix of the hypergraph.

An Equitable partition leads to a quotient matrix of matrices associated with graphs and hypergraphs \cite{Cvetkovic-1980,atik-equitable,agt-godsil-royle,unit,Benjamin-webb-gen-equitable-2019,Sarkar-banerjee-2020}. Each eigenvalue of the quotient matrix is also an eigenvalue of the original matrix \cite{Cvetkovic-1980}. The identification of all the vertices within the same unit results in the unit-contraction of the hypergraph. Specific matrices associated with hypergraphs often have a quotient matrix due to an equitable partition associated with the unit-contraction of the hypergraph. Thus, the eigenvalues of the specific matrices related to the unit-contraction of a hypergraph are also eigenvalues of matrices associated with the original hypergraph (see \cite{unit}). Here we observe a similar fact for the rank of the edge-vertex incidence matrix of a hypergraph. We show that the incidence matrices of the unit-contraction of a hypergraph, and the original hypergraph have the same rank (see \Cref{quotient-rank}).

The dual of a hypergraph $H$ is a hypergraph $H^*$ such that the vertex set of $H^*$ is the set of hyperedges in $H$. The edge set of $H^*$ has a bijection with the vertex set of $H$. The edge-vertex incidence matrix of $H$ is the vertex-edge incidence matrix of $H^*$. Thus, the results we have concluded for the edge-vertex incidence matrix have their analogous version for the vertex-edge incidence matrix (see \Cref{thm-equal partition of vertices} and \Cref{ratio-r-edge}). In a nutshell, we study how the rank and the null space of the incidence matrices of hypergraphs are related to the structures of the hypergraph.

Like graphs, various matrices associated with hypergraphs and their spectra are used to study hypergraphs \cite{Bretto-hypergraph,Sarkar-banerjee-2020,hg-mat,Swarup-panda-2022-hypergraph,equitable-me-2024,Trevisan_signless,invariantsubspaces-me}.
In \cite{unit}, it has been established that certain symmetric sub-structures of hypergraph are manifested in the eigenvalues of some matrices associated with hypergraph. Here, we show that some of these eigenvalues of some variations of the adjacency matrix associated with the hypergraph can be represented in terms of the columns of the incidence matrix of the hypergraph. There are multiple variations of adjacency matrices associated with hypergraphs in literature \cite{hg-mat,Bretto-hypergraph,my1st}. Unlike a graph, a hypergraph cannot be reconstructed from its adjacency matrices. That is, no variation of the adjacency matrix can encode the complete information of a hypergraph. Usually, each variation of adjacency can be associated with an edge-weight. For instance, for the adjacency described in \cite{Bretto-hypergraph}, the edge weight $w(e)=1$ for all the hyperedges $e$ in the hypergraph. For two distinct vertices $u$ and $v$, the $(u,v)$-th entry of the adjacency is the sum of the hyperedge weight over the collection of all the hyperedges that contain both $u$ and $v$. Similarly, for the adjacency described in \cite{hg-mat}, the weight $w(e)=\frac{1}{|e|-1}$ for all hyperedges $e$ in the hypergraph. In this work, we show that specific eigenvalues of these variations of adjacency can be expressed in terms of the edge-vertex incidence matrix and the hyperedge weight $w$ associated with the adjacency matrix.

\section{The null spaces of Incidence matrices of a hypergraph}
 A hypergraph $H$ is an ordered pair of sets $(V(H),E(H))$. Here $V(H)$ is a non-empty set, called the vertex set of the hypergraph $H$, and each element $v(\in V(H))$ is called a vertex in $H$. The set $E(H)$, called the hyperedge set in $H$, is such that each element $e(\in E(H))$ is a non-empty subset of the vertex set $V(H)$. Each element of $E(H)$ is called a hyperedge in $H$. The \emph{edge-vertex incidence matrix} $B_H=[b_{ev}]_{e\in E(H),v\in V(H)}$ of a hypergraph $H$ is defined by $b_{ev}=1$ if $v\in e$, and otherwise $b_{ev}=0$. The \emph{vertex-edge incidence matrix} $I_H$ of $H$ is the transpose of the edge-vertex incidence matrix. That is, $I_H=B_H^T$.

A hypergraph $H$ is called $k$-uniform if the cardinality of $e$ is $|e|=k$ for all $e\in E(H)$. A $2$-uniform hypergraph is called a \emph{graph}.
A \emph{cycle} of length $n$ is a graph $C_n$ with the vertex set $V(C_n)=[n]=\{i\in\mathbb{N}:i\le n\}$ and the hyperedge set $E(C_n)=\{e_i=\{i,i+1\}:i\in[n-1]\}\cup\{e_n=\{n,1\}\}$. Now, we recall an interesting fact about the rank of the edge-vertex incidence matrix from \cite[P.37, Exercise-7]{Brualdi-comb-mat}.
\begin{fac}[\cite{Brualdi-comb-mat}]\label{fact-cycle}
    Let $G$ be a connected graph with $n$ vertices and $n$ edges. The graph $G$ contains a unique cycle $C$. The edge-vertex incidence matrix $B_G$ of $G$ has rank $n$ if the length of $C$ is odd. If the length of $C$ is even, then $B_C$ has rank $n-1$. 
\end{fac}
Now, we explore if a similar fact is true for hypergraphs.
In the following definition, we describe a $k$-uniform hypergraph that coincides with the cycle in the $k=2$ case.
\begin{df}[$k$-uniform cycle]\label{$k$-uniform cycle} For some natural number $k(\ge 2)$, the \emph{$k$-uniform cycle} $C_n^k$ is a $k$-uniform hypergraph with the vertex set $V(C_n^k)=\mathbb{Z}_n$ with $n\ge k$, and the hyperedge set $E(C_n^k)=\{e_i:i\in \mathbb{Z}_n\}$, where $e_i=\{i,i+1,i+2,\ldots,i+(k-1)\}$.
\end{df}
In this definition, we use $\mathbb{Z}_n$ instead of $[n]$ to take advantage of its cyclic group structure. This structure allows us to use the relation $n+1=1$ within $\mathbb{Z}_n$. As a result, any hyperedge in $C_n^k$ has the form $\{l, l+1, l+2, \ldots, l+k-1\}$ for some $l \in \mathbb{Z}_n$, and we denote it by $e_l$. It is also worth noting that $C_n^2 = C_n$.
\begin{exm}[$4$-uniform cycle of length $8$, $C_8^4$]\label{c-4-8}\rm
The vertex set of $C_8^4$ is $V(C_8^4)=\mathbb{Z}_8$, and the hyperedge set is $E(C_8^4)=\{e_1=\{1,2,3,4\},e_2=\{2,3,4,5\},e_3=\{3,4,5,6\},e_4=\{4,5,6,7\},e_5=\{5,6,7,8\},e_6=\{6,7,8,1\},e_7=\{7,8,1,2\},e_8=\{8,1,2,3\}\} $.
\end{exm}
For the cycle graph $C_{2n}$, the vector $x:[2n]\to\{-1,1\}$, defined by $x(i)={(-1)}^{i}$, for all $i\in [2n]$, belongs to the null space of the edge-vertex incidence matrix $B_{C_{2n}}$. Since each edge of $C_{2n}$ consists of an odd and an even, therefore, for all $e\in E(C_{2n})$, the $e$-th entry of $B_{C_{2n}}x$ is $(B_{C_{2n}}x)(e)=1-1=0$. A similar result holds for the $k$-uniform cycle. For any natural number $k>1$, for a $k$-th root of unity $\omega$, we define $x_{\omega}:\mathbb{Z}_n\to\mathbb{C}$ as $x_{\omega}(i) =\omega^i$ for all $i\in\mathbb{Z}_n$.
\begin{prop}\label{x-w-prop}
    For some natural number $k>1$, if $n$ is a multiple of $k$, then $B_{C_n^k}x_{\omega}=0$, where $\omega(\ne 1)$ is a $k$-th root of unity.
\end{prop}
\begin{proof} 
Given $k$ is a multiple of $n$, and $\omega(\ne 1)$ is a $k$-th root of unity, we have $\omega^n=1$, and $\omega^{n+i}=\omega^i$ for $i=0,1,\ldots,k-1$.
    Since for each $e\in E(C_n^k)$, the hyperedge $e=\{i,i+1,i+2,\ldots,i+(k-1)\}$ for some $i\in\mathbb{Z}_n$, the $e$-th entry of the vector $B_{C_n^k}x_{\omega}$ is $(B_{C_n^k}x_{\omega})(e)=\sum\limits_{i\in e}x_{\omega}(i)=\omega^{i}\sum\limits_{j=0}^{n-1}\omega^{j}=0$. Therefore, $B_{C_n^k}x_{\omega}=0$.
\end{proof}
 By the \Cref{fact-cycle}, for $C_{n}=C^2_n$ the rank of $B_{C_n}$ is $n-1$ if $n$ is even. As \Cref{x-w-prop} suggests, if  $n$ is a multiple of $k$, then the rank of $B_{C_n^k} $ is at most $n-1$. The next example will justify that for the hypergraph case, the rank of $B_{C_n^k}$ is not exact $n-1$, and it is at most $n-1$.
\begin{exm}\rm\label{ex-nulity>2}
    Consider the $4$ uniform cycle of length $8$ (defined in \Cref{c-4-8}). As the \Cref{x-w-prop} suggests, for $\omega=e^{\iota\frac{2\pi}{8}}$, a $4$-th root of unity, $B_{C_8^4}x_{\omega}=0$. Now, for the vector $y:\mathbb{Z}_8\to\{-1,1\}$ defined by $y(2i-1)=-1$, and $y(2i)=1$ for all $i=1,2,\ldots,8$. Since each hyperedge of $C_8^4$ contains two even and two odd numbers, the vector  $B_{C_8^4}y=0$. Now, with respect to the usual inner product on $\mathbb{C}^8$, the inner product $\langle x_{\omega},y\rangle=0$. Therefore, $x_\omega$ and $y$ are the two linearly independent vectors in the null space of $B_{C_8^4}$.
\end{exm}
In the \Cref{ex-nulity>2}, the vector $y=x_{\omega'}$, where $\omega'=-1$ is a $4$-th root of unity. This fact motivates the following result. 
\begin{prop}\label{prop-cycle}
    For some natural number $k>1$, if $n$ is a multiple of $k$, then the rank of $B_{C_n^k}$ is at most $n-k+1$.
\end{prop}
\begin{proof}
    By the \Cref{x-w-prop}, if $\omega_k=e^{\iota\frac{2\pi}{k}}$, and $\{\omega_k^0, \omega_k^1,\ldots,\omega_k^{k-1}\}$ are the $k$ distinct $k$-th root of unity then $\{x_{\omega_k^i}:i=1,2,\ldots,k-1\}$ are $k-1$ vectors in the null space of $B_{C_n^k}$. The matrix $[\omega_k^{ij}]_{i,j\in[k]}$ has Vandermonde determinant, and $\omega_k^j\ne\omega_k^{j'}$ for two distinct $i,j\in [k]$. Consequently, the collection of vector is $\{x_{\omega_k^i}:i=1,2,\ldots,k-1\}$, and the dimension of the null space of $B_{C_n^k}$ is at least $k-1$, and the rank of the matrix is at most $n-k+1$.
\end{proof}
By the \Cref{fact-cycle}, if $n$ is odd, then there is no non-zero vector in the null space of $B_{C_n}$. The \Cref{x-w-prop}, and the \Cref{prop-cycle} implies that when $n$ is a multiple of $k$, then there are at least $k-1$ linearly independent vectors in the null space of $B_{C_n^k}$. Now, we show that even if $n$ is not a multiple of $k$, the dimension of the null space of $B_{C_n^k}$ is not necessarily $0$.
\begin{exm}\label{ex-c46}
    Consider the $4$-uniform cycle on $6$ vertices $C_6^4$. Here $6$ is not a multiple of $4$. For the vector $y:\mathbb{Z}_6\to\{1,-1\}$ as $y(i)=(-1)^i$, then since each hyperedge contains two even and two odd numbers, $C_6^4y=0$.
\end{exm}
In the \Cref{ex-c46}, though $6$ is not a multiple of $4$ but their greatest common divisor, $gcd(6,4)=2$, and $-1$ is the $2$-nd root of unity. This fact motivates the following result.
\begin{thm}\label{gcd}
     For some natural number $k>1$, and for any natural number $n(\ge k)$, if the greatest common divisor, $gcd(k,n)=r$, then the rank of $B_{C_n^k}$ is at most $n-r+1$.
\end{thm}
\begin{proof}
    Let $\omega_r=e^{\iota\frac{2\pi}{r}}$. The complete set of the $r$-th root of unity is $\{\omega_r^i:i=1,2,\ldots,r\}$.  Consider the vectors $x_{\omega_r^i}:\mathbb{Z}_n\to\mathbb{C}$, defined by $x_{\omega_r^i}(j)=( \omega_r^i)^j$ for all $j\in\mathbb{Z}_n$, and for all $i=1,2,\ldots,r$. For any $e\in E(C_n^k)$, we have $e=\{l,l+1,\ldots,l+k-1\}$ for some $l\in\mathbb{Z}_n$. Consequently, for all $e\in E(C_n^k)$
    \begin{align*}
        &(B_{C_n^k}x_{\omega_r^i})(e)=\sum\limits_{s=0}^{k-1}(\omega_r^i)^{l+s}.
    \end{align*}
    Since $gcd(k,n)=r$, there exist two natural numbers $p,q$ with $gcd(p,q)=1$, and $k=pr,n=qr$. Being $\omega^i_r(\ne 1)$ is an $r$-th root of unity, $(\omega^i_r)^n=1$, and the sum $\sum\limits_{s=0}^{r-1}(\omega_r^i)^s=0$. Consequently,
    \begin{align*}
        &\sum\limits_{s=0}^{k-1}(\omega_r^i)^{l+s}=\sum\limits_{j=0}^{p-1}(\omega_r^i)^{l+j}\sum\limits_{s=0}^{r-1}(\omega_r^i)^s=0.
    \end{align*}
    Therefore, $B_{C_n^k}x_{\omega_r^i}=0$.
\end{proof}
If $n$ is even, then for $k=2$, by the \Cref{fact-cycle}, the dimension of the null space of $B_{C_n^k}$ is exactly $1$, but by the \Cref{prop-cycle} for $k>2$, if $n$ is a multiple of $k$, then the dimension may be more than $1$. \Cref{gcd} shows that even if $n$ is not a multiple of $k$, there may be non-zero vectors in the null space of $B_{C_n^k}$. That is, a hypergraph induces more vectors in the null space of its edge-vertex incidence matrix compared to a graph. In this section, we explore the properties of hypergraphs that lead to these additional vectors in the null space of their incidence matrices. 

\begin{df}[Sub-hypergraph induced by a set $U$ \cite{Berge-hypergraph}]
    Let $H$ be a hypergraph, and $U\subseteq V(H)$. If a hypergraph $H_U$ is such that $V(H_U)=U$, and $E(H_U)=\{e\cap U:e\in E(H), e\cap U\ne\emptyset\}$, then $H_U$
    is called the \emph{sub-hypergraph of $H$ induced by $U$}.
\end{df}
\begin{exm}\label{ex-induced}
    Consider the hypergraph $H$ with $V(H)=\{1,2,3,4,5,6,7\}$, and $E(H)=\{e_1=\{1,2,3,7\},e_2=\{2,3,4,7\},e_3=\{3,4,5,7\},e_4=\{4,5,6,7\},e_5=\{5,6,1,7\}, e_6=\{6,1,2,7\}\}$. The sub-hypergraph of $H$ induced by $U=\{1,2,3,4,5,6\}$ is a $3$-uniform cycle on $ 6$ vertices, $C_6^3$.
\end{exm}
Given any hypergraph $H$. For any $U\subseteq V(H)$ and any vector $y:U\to\mathbb{C}$, the \emph{extension} of $y$ in $V(H)$ is the vector $y':V(H)\to \mathbb{C}$ defined as $ y'(i)=y(i)$ if $i\in U$, and $y'(i)=0$ if $i\in V(H)\setminus U$.   

Suppose that a vector $y:\mathbb{Z}_6\to\mathbb{C}$ is defined as $y(j)=\omega^j$ for all $j\in\mathbb{Z}_6$, where $\omega=e^{\iota\frac{2\pi}{3}}$ is a $3$-rd root of unity. For each $e_l=\{l,l+1,l+2\}\in E(C_6^3)$, we have $(B_{C_6^3}y)(e_l)=\omega^l(1+\omega+\omega^2)=0$. Consider the hypergraph $H$, described in the \Cref{ex-induced}. We extend the vector $y$ as the vector $y':V(H)\to\mathbb{C}$ as $y'(j)=y(j)$ if $j\in V(C_6^3)$, otherwise $y'(j)=0$ for all $j\in V(H)\setminus V(C_6^3)$. For any $e\in E(H)$ with $e\cap U\ne\emptyset$, we have $ (B_{H}y')(e)=(B_{C_6^3}y)(e\cap U)=0$, and if $e\cap U=\emptyset$, then $ (B_{H}y)(e)=\sum\limits_{j\in e}y(e)=0$. Consequently, $B_{H}y'=0$. This fact motivates the following result. 
\begin{thm}\label{thm-induced}
  Let \( H \) be a hypergraph. Suppose that \( U \subseteq V(H) \) and \( H_U \) is the sub-hypergraph of \( H \) induced by \( U \). If a vector \( y: U \to \mathbb{C} \) belongs to the null space of \( B_{H_U} \), then its extension \( y' \) in \( V(H) \) belongs to the null space of \( B_H \).
\end{thm}
\begin{proof}
    For all $e\in E(H)$, either $e\cap U=\emptyset$ or $e\cap U\ne\emptyset$. If $e\cap U=\emptyset$, then $y'(i)=0$ for all $i\in e$. Consequently, $(B_Hy')(e)=\sum\limits_{i\in e}y'(i)=0$. If $e\cap U\ne\emptyset$, then $ e\cap U\in E(H_U)$. Therefore, $ (B_Hy')(e)=\sum\limits_{i\in e}y'(i)=\sum\limits_{i\in e\cap U}y(i)=(B_{H_U}y)(e\cap U)=0$. Consequently, $B_{H}y'=0 $.
\end{proof}
The \Cref{thm-induced} and the \Cref{gcd} lead us to the following Corollary.
\begin{cor}\label{cycle-induced}
    If a hypergraph $H$ contains any $k$-uniform cycle on $n$ vertices $C_n^k$ as a sub-hypergraph induced by some $U\subset V(H)$ for some $k>1$ with $gcd(k,n)=r$, then the dimension of the null space of $B_H$ is at least $r-1$.
\end{cor}
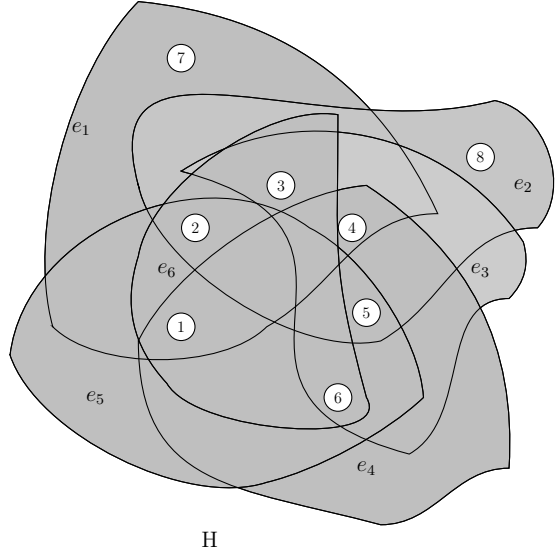
\begin{figure}[H]
    \centering
 \begin{tikzpicture}[scale=0.75]
		\node [style=none] (0) at (-2.25, 7.75) {};
		\node [style=none] (1) at (-3.75, 2) {};
		\node [style=none] (2) at (3, 4) {};
		\node [style=none] (3) at (0, 2) {};
		\node [style=new style 0,scale=0.5] (4) at (-1.5, 2) {1};
		\node [style=new style 0,scale=0.5] (5) at (-1.25, 3.75) {2};
		\node [style=new style 0,scale=0.5] (6) at (0.25, 4.5) {3};
		\node [style=new style 0,scale=0.5] (7) at (1.5, 3.75) {4};
		\node [style=none] (8) at (4, 6) {};
		\node [style=none] (9) at (-2.25, 4.5) {};
		\node [style=none] (10) at (4.75, 3.75) {};
		\node [style=none] (11) at (2, 1.75) {};
		\node [style=new style 0,scale=0.5] (12) at (1.25, 0.75) {6};
		\node [style=none] (13) at (4.5, 3.5) {};
		\node [style=none] (14) at (-1.5, 4.75) {};
		\node [style=none] (15) at (4.25, 2.5) {};
		\node [style=none] (16) at (2.5, -0.25) {};
		\node [style=new style 0,scale=0.5] (17) at (1.75, 2.25) {5};
		\node [style=none] (18) at (1.75, 4.5) {};
		\node [style=none] (19) at (-2.25, 1.75) {};
		\node [style=none] (20) at (4.25, -0.5) {};
		\node [style=none] (21) at (2, -1.5) {};
		\node [style=none] (22) at (0.75, 3.75) {};
		\node [style=none] (23) at (-4.5, 1.5) {};
		\node [style=none] (24) at (2.75, 0.75) {};
		\node [style=none] (25) at (0, -0.75) {};
		\node [style=none] (26) at (1.25, 5.75) {};
		\node [style=none] (27) at (1.75, 0.75) {};
		\node [style=none] (28) at (-1.75, 1) {};
		\node [style=none] (29) at (-2.25, 3.25) {};
		\node [style=new style 0,scale=0.5] (30) at (-1.5, 6.75) {7};
		\node [style=new style 0,scale=0.5] (31) at (3.75, 5) {8};
		\node [style=none,scale=0.7] (32) at (-3.25, 5.5) {$e_1$};
		\node [style=none,scale=0.7] (33) at (4.5, 4.5) {$e_2$};
		\node [style=none,scale=0.7] (34) at (3.75, 3) {$e_3$};
		\node [style=none,scale=0.7] (35) at (1.75, -0.5) {$e_4$};
		\node [style=none,scale=0.7] (36) at (-3, 0.75) {$e_5$};
		\node [style=none,scale=0.7] (37) at (-1.75, 3) {$e_6$};
		\node [style=none,scale=0.7] (38) at (-1, -1.75) {H};
		\draw[fill=gray!50!white,opacity=0.50] (1.center)
			 to [bend left, looseness=0.75] (0.center)
			 to [bend left] (2.center)
			 to [in=30, out=-180] (3.center)
			 to [bend left=45, looseness=0.75] cycle;
		\draw[fill=gray!45!white,opacity=0.50] (9.center)
			 to [in=-165, out=105, looseness=1.25] (8.center)
			 to [bend left=60] (10.center)
			 to [in=30, out=-180] (11.center)
			 to [bend left=45, looseness=0.75] cycle;
		\draw[fill=gray!40!white,opacity=0.50] (14.center)
			 to [bend left=45] (13.center)
			 to [bend left] (15.center)
			 to [in=30, out=-180] (16.center)
			 to [in=-15, out=165, looseness=1.75] cycle;
		\draw[fill=gray!50!white,opacity=0.50] (19.center)
			 to [bend left, looseness=0.75] (18.center)
			 to [bend left] (20.center)
			 to [in=0, out=-180] (21.center)
			 to [in=270, out=165, looseness=1.25] cycle;
		\draw[fill=gray!50!white,opacity=0.50] (23.center)
			 to [bend left=60] (22.center)
			 to [bend left, looseness=0.75] (24.center)
			 to [bend left=15, looseness=0.75] (25.center)
			 to [bend left=45, looseness=0.75] cycle;
		\draw[fill=gray!50!white,opacity=0.50] (27.center)
			 to [in=-60, out=-60, looseness=0.75] (28.center)
			 to [bend left] (29.center)
			 to [bend left=45, looseness=0.75] (26.center)
			 to [in=105, out=-90, looseness=1.25] cycle;
 \draw (1.center)
			 to [bend left, looseness=0.75] (0.center)
			 to [bend left] (2.center)
			 to [in=30, out=-180] (3.center)
			 to [bend left=45, looseness=0.75] cycle;
		\draw (9.center)
			 to [in=-165, out=105, looseness=1.25] (8.center)
			 to [bend left=60] (10.center)
			 to [in=30, out=-180] (11.center)
			 to [bend left=45, looseness=0.75] cycle;
		\draw (14.center)
			 to [bend left=45] (13.center)
			 to [bend left] (15.center)
			 to [in=30, out=-180] (16.center)
			 to [in=-15, out=165, looseness=1.75] cycle;
		\draw (19.center)
			 to [bend left, looseness=0.75] (18.center)
			 to [bend left] (20.center)
			 to [in=0, out=-180] (21.center)
			 to [in=270, out=165, looseness=1.25] cycle;
		\draw (23.center)
			 to [bend left=60] (22.center)
			 to [bend left, looseness=0.75] (24.center)
			 to [bend left=15, looseness=0.75] (25.center)
			 to [bend left=45, looseness=0.75] cycle;
		\draw (27.center)
			 to [in=-60, out=-60, looseness=0.75] (28.center)
			 to [bend left] (29.center)
			 to [bend left=45, looseness=0.75] (26.center)
			 to [in=105, out=-90, looseness=1.25] cycle;
 	\node [style=new style 0,scale=0.5] (4) at (-1.5, 2) {1};
		\node [style=new style 0,scale=0.5] (5) at (-1.25, 3.75) {2};
		\node [style=new style 0,scale=0.5] (6) at (0.25, 4.5) {3};
		\node [style=new style 0,scale=0.5] (7) at (1.5, 3.75) {4};
  \node [style=new style 0,scale=0.5] (12) at (1.25, 0.75) {6};
  \node [style=new style 0,scale=0.5] (17) at (1.75, 2.25) {5};
  \node [style=new style 0,scale=0.5] (30) at (-1.5, 6.75) {7};
		\node [style=new style 0,scale=0.5] (31) at (3.75, 5) {8};
		\node [style=none,scale=0.7] (32) at (-3.25, 5.5) {$e_1$};
		\node [style=none,scale=0.7] (33) at (4.5, 4.5) {$e_2$};
		\node [style=none,scale=0.7] (34) at (3.75, 3) {$e_3$};
		\node [style=none,scale=0.7] (35) at (1.75, -0.5) {$e_4$};
		\node [style=none,scale=0.7] (36) at (-3, 0.75) {$e_5$};
		\node [style=none,scale=0.7] (37) at (-1.75, 3) {$e_6$};
\end{tikzpicture}
   \caption{A hypergraph $H$ with \( V(H) = \{n \in \mathbb{N} : n \leq 8\} \) and \( E(H) = \{e_1 = \{1, 2, 3, 4, 7\}, e_2 = \{2, 3, 4, 5, 8\}, e_3 = \{3, 4, 5, 6\}, e_4 = \{4, 5, 6, 1\}, e_5 = \{5, 6, 1, 2\}, e_6 = \{6, 1, 2, 3\}\} \). The subset \( U = \{n \in \mathbb{N} : n \leq 6\} \subset V(H) \) of the vertex set the induced the sub-hypergraph \( H_U=C_6^4  \). The pairwise-disjoint collection of vertices $W=\{1,3,5\}$, $U=\{2,4\}$, and $V=\{7,8\}$ are such that $(|e\cap U|-|e\cap V|):(e\cap W)=\frac{1}{2}$ for all $e\in E(H)$.}
    \label{fig:gcd-ex}
\end{figure}

\begin{exm}\label{gcd-r-example}
   Let \( H \) be a hypergraph where \( V(H) = \{n \in \mathbb{N} : n \leq 8\} \) and \( E(H) = \{e_1 = \{1, 2, 3, 4, 7\}, e_2 = \{2, 3, 4, 5, 8\}, e_3 = \{3, 4, 5, 6\}, e_4 = \{4, 5, 6, 1\}, e_5 = \{5, 6, 1, 2\}, e_6 = \{6, 1, 2, 3\}\} \) (see the \Cref{fig:gcd-ex}). If we consider the subset \( U = \{n \in \mathbb{N} : n \leq 6\} \subset V(H) \) induce the sub-hypergraph \( H_U=C_6^4  \). In this case, the greatest common divisor of 4 and 6 is 2. Now, consider the vector \( y: V(H) \to \{0, 1, -1\} \) defined by \( y(i) = (-1)^i \) for all \( i \in [6] \), and $y(i)=0$ for \( i \in \{7, 8\} \). As suggested by \Cref{thm-induced} and \Cref{cycle-induced}, we find that \( B_Hy = 0 \).
\end{exm}
It is intriguing to note that in the previous example, the pair of disjoint subsets $y^{-1}(1)=\{2,4,6\}$ and $y^{-1}(-1)=\{1,3,5\}$ are such that for all $e\in E(H)$ we have $|e\cap  y^{-1}(1)|=|e\cap y^{-1}(-1)| $. This motivates the following notion.

\begin{df}[Equal partition of hyperedges]\label{equal partition defn}
   Given two disjoint subsets $U=\{u_1,u_2,\ldots,u_p\}$ and $V=\{v_1,v_2,\ldots,v_q\}$ of the vertex set $V(H)$ of a hypergraph $H$, if
$|U\cap e|=|V\cap e| $ for all $e\in E(H)$, then we refer to the pair $U$, and $V$ as an \emph{equal partition of hyperedges} in $H$.
\end{df}
\begin{exm}
  Let $H$ be a hypergraph with $V(H)=\{1,2,3,4,5\}$ and $E(H)=\{e_1,e_2,e_3\}$, where $e_1=\{1,2,3,5\}$, $e_2=\{1,3,4,5\}$, and $e_3=\{1,2,4,5\}$. Consider the pair of disjoint subsets of vertices $U=\{1,5\}$ and $V=\{2,3,4\}$. Since $|U\cap e_i|=2=|V\cap e_i|$ for all $i=1,2,3$, the pair of sets $U$ and $V$ forms an equal partition of hyperedges in $H$.
\end{exm}
In the next Theorem, we show that an equal partition of hyperedges causes a vector in the null space of the edge-vertex incidence matrix of the hypergraph. Given any hypergraph $H$, and $U\subseteq V(H)$, the characteristic function $\chi_U:V(H)\to\{0,1\}$ of $U$ is such a function that $\chi_U(v)=1$ if $v\in U$, otherwise $\chi_U(v)=0$.
\begin{thm}\label{thm-equal-partition}
   Let $H$ be a hypergraph. The pair of subsets of vertices $U$ and $V$ is an equal partition of hyperedges in $H$ if and only if $B_H(\chi_U-\chi_V)=0$.
\end{thm}
\begin{proof}
Let $U=\{u_1,u_2,\ldots,u_p\}$ and $V=\{v_1,v_2,\ldots,v_q\}$ form an equal partition of hyperedges in $H$. Consequently, $|e\cap U|=|e\cap V|$ for all $e\in E(H)$, and that leads to $(B_H(\chi_U-\chi_V))(e)=|e\cap U|-|e\cap V|=0$ for all $e\in E(H)$.

Conversely, suppose that $B_H(\chi_U-\chi_V)=0$. For all $e\in E(H)$, since $0=(B_H(\chi_U-\chi_V))(e)=|e\cap U|-|e\cap V|$, we have $|e\cap U|=|e\cap V|$.
\end{proof}
 Consider a hypergraph $H$ with the vertex set $V(H)=\{1,2,3,4,5\}$ and the hyperedge set $E(H)=\{e_1=\{1,3,4\},e_2=\{2,4,5\}\}$. Here we have a pair of disjoint subsets $U=\{1,2\}$ and $V=\{3,4,5\}$ such that the ratio $|e\cap U|:|e\cap V|=1:2$ for all $e\in E(H)$.
For the vector $x=2\chi_U-\chi_V$, it holds that $B_Hx=0$. This example suggests the potential for further extending \Cref{thm-equal-partition}, which we will explore in the following result.
 \begin{thm}\label{ratio-r}
     Let $H$ be a hypergraph. There is a pair of disjoint collections of vertices $U$ and $V$ with the ratio $|e\cap U|:|e\cap V|=r$ for all $e\in E(H)$ if and only if $ B_H(\chi_U-r\chi_V)=0$.
 \end{thm}
 \begin{proof}
    Suppose that the ratio $|e\cap U|:|e\cap V|=r$ for all $e\in E(H)$. Therefore, for all $e\in E(H)$, we have $|e\cap U|-r|e\cap V|=0 $, and consequently, $(B_H(\chi_U-r\chi_V))(e)=|e\cap U|-r|e\cap V|=0$. 
   
    Conversely, suppose that $B_H(\chi_U-r\chi_V)=0$. Since $(B_H(\chi_U-r\chi_V))(e)=|e\cap U|-r|e\cap V|$ for all $e\in E(H) $, it holds that $|e\cap U|:|e\cap V|=r$ for all $e\in E(H)$.
 \end{proof}
Suppose that $H$ is a hypergraph, with the vertex set $V(H)=\{1,2,3,4,5,6\}$, and the hyperedge set $E(H)=\{e_1=\{1,3,4\},e_2=\{2,4,5\},e_3=\{1,3,4,5,6\}\}$. Consider the disjoint sub-collection of vertices $U=\{3,4,5\}$, $V=\{6\}$, and $W=\{1,2\}$. Here $(|e\cap U|-|e\cap V|):|e\cap W|=2:1$, and $2\chi_W-(\chi_U-\chi_V)$ is a vector belongs to the null space of $B_H$. This instance motivates another hypergraph structure related to the null space of the incidence matrix.
\begin{thm}\label{uvw}
    Let $H$ be a hypergraph. Suppose that $U$, $V$, and $W$ are three pair-wise disjoint subsets of the vertex set $V(H)$, with $(|e\cap U|-|e\cap V|):|e\cap W|=r$ for all $e\in E(H)$ if and only if $ r\chi_{W}-(\chi_U-\chi_V)$ belongs to the null space of $B_H$.
\end{thm}
\begin{proof}
    Suppose that $(|e\cap U|-|e\cap V|):|e\cap W|=r$ for all $e\in E(H)$. Therefore, $B_H(r\chi_{W}-(\chi_U-\chi_V))(e)=r|e\cap W|-(|e\cap U|-|e\cap V|)=0$ for all $e\in E(H)$. Consequently,  $B_H(r\chi_{W}-(\chi_U-\chi_V))=0$. 
    
    Conversely, suppose that $ r\chi_{W}-(\chi_U-\chi_V)$ belongs to the null space of $B_H$. Thus, for all $e\in E(H)$ it holds that $0=B_H(r\chi_{W}-(\chi_U-\chi_V))(e)=r|e\cap W|-(|e\cap U|-|e\cap V|)$. Therefore, $(|e\cap U|-|e\cap V|):|e\cap W|=r$ for all $e\in E(H)$.
\end{proof}
Consider the hypergraph $H$ described in the \Cref{gcd-r-example} (\Cref{fig:gcd-ex}). The pairwise-disjoint collection of vertices $W=\{1,3,5\}$, $U=\{2,4\}$, and $V=\{7,8\}$ are such that $(|e\cap U|-|e\cap V|):(e\cap W)=\frac{1}{2}$ for all $e\in E(H)$. Therefore, by the \Cref{uvw}, the vector $\frac{1}{2}\chi_{W}-(\chi_U-\chi_V)$ belongs to the null space of $B_H$.

Consider a hypergraph $H$ with vertex set $V(H)=\{1,2,3,4,5,6\}$ and hyperedge set $E(H)=\{e_1=\{1,5,3,6\},e_2=\{1,2\},e_3=\{2,6\},e_4=\{3,4\},e_5=\{4,5,6\}\}$. For the vector $y:V(H)\to\mathbb{R}$ defined by $y(1)=-y(2)=y(6)=1$, $-y(3)=y(4)=\frac{1}{2}$, and $y(5)=-\frac{3}{2}$, we have $B_Hy=0$. It is intriguing to note that $y$ belongs to the null space of $B_H$ because any hyperedges $e\in E(H)$ satisfy the following equation:
$$|e\cap \{1,6\}|-|e\cap \{2\}|+\frac{1}{2}|e\cap \{4\}|-\frac{1}{2}|e\cap\{3\}|-\frac{3}{2}|e\cap\{5\}|=0.$$ This fact motivates a more general scenario than that of the \Cref{ratio-r} and the \Cref{uvw}.

\begin{thm}\label{gen-ker}
   Let $H$ be a hypergraph. For any hyperedge $e\in E(H)$, a collection of pairwise disjoint subsets of the vertex set $U_1,\ldots,U_n$ satisfy the equation $\sum_{i=1}^nc_i|e\cap U_i|=0$ if and only if $B_H(\sum\limits_{i=1}^nc_i\chi_{U_i})=0$ where $c_i\in \mathbb{C}$ for some $i=1,\ldots,n$.
\end{thm}
\begin{proof}
   Let us assume we have a collection of subsets \( U_1, \ldots, U_n \) that are pairwise disjoint and satisfy the equation \(\sum\limits_{i=1}^n c_i |e \cap U_i| = 0\) for any hyperedge \(e \in E(H)\). This implies that for any hyperedge \(e\), the expression \((B_H(\sum\limits_{i=1}^n c_i \chi_{U_i}))(e)=\sum\limits_{i=1}^n c_i |e \cap U_i| = 0\). Therefore, we can conclude that \(B_H(\sum\limits_{i=1}^n c_i \chi_{U_i}) = 0\).

Conversely, if \(B_H(\sum\limits_{i=1}^n c_i \chi_{U_i}) = 0\), then for every hyperedge \(e \in E(H)\), it must be true that \(\sum\limits_{i=1}^n c_i |e \cap U_i| = (B_H(\sum\limits_{i=1}^n c_i \chi_{U_i}))(e) = 0\).
\end{proof}
Given any vertex $v$ in a hypergraph $H$, the \emph{star of the vertex} $v$ is $E_v(H)=\{e\in E(H):v\in e\}$.
Now we revisit the concept of a unit in a hypergraph as introduced in \cite{unit}. In a hypergraph $H$, units are the maximal collections of vertices with the same stars. Unit is another hypergraph structure, which is responsible for the vectors in the null space of the edge-vertex incidence matrix. 
\begin{df}[Unit]{\rm\cite[Definition 3.1]{unit}}
    Let \(H\) be a hypergraph. Consider the equivalence relation \(\mathcal{R}_u(H)\) on the vertex set \(V(H)\) given by \[\mathcal{R}_u(H) = \{(u, v) \in V(H) \times V(H) : E_u(H) = E_v(H)\}.\] Each equivalence class under \(\mathcal{R}_u(H)\) is known as a \emph{unit}. For every unit \(W_E \subseteq V(H)\), there exists a subset \(E \subseteq E(H)\) such that \(E_v(H) = E\) for all \(v \in W_E\). This subset \(E\) is called the \emph{generator} of the unit \(W_E\).
\end{df}
 We denote the complete collection of units in $H$ as $\mathfrak{U}(H)$. Given a hypergraph $H$, the \emph{unit-contraction} of $H$ is a hypergraph $H/\mathcal{R}_u(H)$ with \[V(H/\mathcal{R}_u(H))=\mathfrak{U}(H),\]
and \[E(H/\mathcal{R}_u(H))=\{\Tilde{e}=\{W_{E_v(H)}:v\in e\}: e\in E(H)\}.\] 

For any $e\in E(H)$, $\Tilde{e}$ is a set, and if $E_u(H)=E=E_v(H)$ for two $u,v\in e$, then $\Tilde{e}$ contains the unit $W_E$ containing $u,v$. However, to avoid possible confusion, it is important to clarify that being a set, $\Tilde{e}$ contains $W_E$ just once and does not contain two distinct instances of $W_E$ for both $u$ and $v$ individually. 
\begin{exm}\label{ex-unit}\rm
    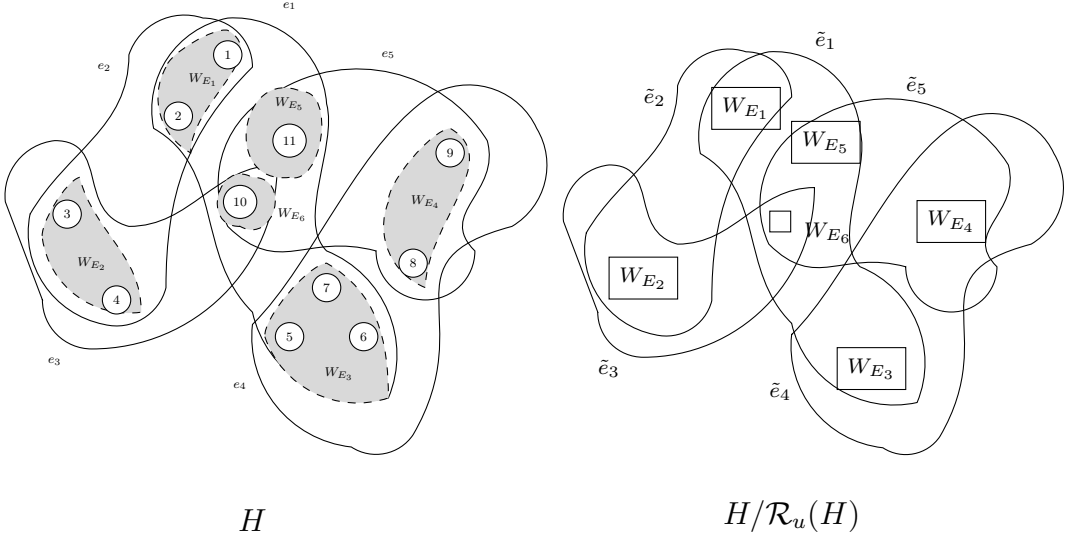
\begin{figure}[H]
        \centering
        \begin{subfigure}{0.49\textwidth}
        \begin{tikzpicture}[scale=0.65]
		\node [style=none] (11) at (4, 5) {};
		\node [style=none] (12) at (6, 4.25) {};
		\node [style=none] (13) at (5.5, 2.25) {};
		\node [style=none] (14) at (0.25, 0.25) {};
		\node [style=none] (15) at (2.25, -2.25) {};
		\node [style=none] (16) at (3.5, -2) {};
		\node [style=none] (17) at (-1.75, 4.25) {};
		\node [style=none] (18) at (-0.25, 6.5) {};
		\node [style=none] (19) at (1.75, 4.75) {};
		\node [style=none] (20) at (0.25, 0.5) {};
		\node [style=none] (21) at (3, -1.25) {};
		\node [style=none] (22) at (1.75, 1.75) {};
		\node [style=none] (24) at (-4.75, 2.75) {};
		\node [style=none] (25) at (-3.75, 4) {};
		\node [style=none] (26) at (-2.25, 2.25) {};
		\node [style=none] (27) at (-4, 0.75) {};
		\node [style=none] (28) at (-3, -0.25) {};
		\node [style=none] (29) at (0.75, 3.5) {};
		\node [style=none] (30) at (-0.25, 2.25) {};
		\node [style=none] (31) at (0.75, 5) {};
		\node [style=none] (32) at (5, 4) {};
		\node [style=none] (33) at (2.75, 1.75) {};
		\node [style=none] (34) at (3.75, 0.75) {};
		\node [style=none] (35) at (4.75, 1.75) {};
		\node [style=none] (36) at (-2.25, 5.75) {};
		\node [style=none] (37) at (-0.75, 6.5) {};
		\node [style=none] (38) at (0.25, 5.5) {};
		\node [style=none] (39) at (-4.25, 2.5) {};
		\node [style=none] (40) at (-2.75, 0.25) {};
		\node [style=none] (41) at (-1.5, 1) {};
		\node [style=none] (42) at (0.5, 0) {};
		\node [style=none] (43) at (1.75, 1.5) {};
		\node [style=none] (44) at (3, -1.25) {};
		\node [style=none] (45) at (3, 2.75) {};
		\node [style=none] (46) at (4.25, 4.25) {};
		\node [style=none] (47) at (3.75, 1) {};
		\node [style=none] (48) at (-4, 2) {};
		\node [style=none] (49) at (-3.25, 3.25) {};
		\node [style=none] (50) at (-2, 0.5) {};
		\node [style=none] (51) at (-1.5, 5.5) {};
		\node [style=none] (52) at (-0.25, 6.25) {};
		\node [style=none] (53) at (-1, 3.75) {};
		\node [style=none] (56) at (0.25, 3.25) {};
		\node [style=none] (57) at (0, 2.25) {};
		\node [style=none] (58) at (1.25, 5) {};
		\node [style=none] (59) at (1, 3.25) {};
		\node [style=none] (67) at (0.75, 2.25) {};
		\node [style=none] (68) at (1, -1.25) {};
		\node [style=none] (69) at (3.75, 2.75) {};
		\node [style=none] (70) at (2, 4.25) {};
		\node [style=none] (71) at (1, 2.75) {};
		\draw (14.center)
			 to [in=-150, out=45, looseness=0.75] (11.center)
			 to [bend left=45] (12.center)
			 to [bend left=45] (13.center)
			 to [in=60, out=-165, looseness=1.50] (16.center)
			 to [bend left=45] (15.center)
			 to [bend left=45] cycle;
		\draw (20.center)
			 to [in=-30, out=135] (17.center)
			 to [bend left=45] (18.center)
			 to [bend left=45] (19.center)
			 to [in=135, out=-75, looseness=0.75] (22.center)
			 to [bend left=45] (21.center)
			 to [bend left=45] cycle;
		\draw (27.center)
			 to (24.center)
			 to [bend left=45] (25.center)
			 to [in=165, out=0, looseness=1.25] (26.center)
			 to [in=180, out=0] (29.center)
			 to [bend left=45] (28.center)
			 to [bend left=45] cycle;
		\draw (33.center)
			 to [in=-45, out=165, looseness=1.25] (30.center)
			 to [bend left=45] (31.center)
			 to [bend left=45] (32.center)
			 to [in=135, out=-75, looseness=1.25] (35.center)
			 to [bend left=45] (34.center)
			 to [bend left=45] cycle;
		\draw (39.center)
			 to [in=-90, out=60] (36.center)
			 to [bend left=45] (37.center)
			 to [bend left=45] (38.center)
			 to [in=90, out=-135] (41.center)
			 to [bend left=45] (40.center)
			 to [bend left=45] cycle;
		\draw [style=new edge style 1,fill=gray!30!white] (42.center)
			 to [bend right=45] (44.center)
			 to [in=-30, out=90] (43.center)
			 to [bend right=45, looseness=0.50] cycle;
		\draw [style=new edge style 1,fill=gray!30!white] (45.center)
			 to [bend right=45] (47.center)
			 to [in=-30, out=90] (46.center)
			 to [bend right=45, looseness=0.50] cycle;
		\draw [style=new edge style 1,fill=gray!30!white] (48.center)
			 to [bend right=45] (50.center)
			 to [in=-75, out=90] (49.center)
			 to [bend right=45, looseness=0.50] cycle;
		\draw [style=new edge style 1,fill=gray!30!white] (51.center)
			 to [bend right=45] (53.center)
			 to [in=-30, out=90] (52.center)
			 to [bend right=45, looseness=0.50] cycle;
		\draw [style=new edge style 1,fill=gray!30!white] (57.center)
			 to [bend left=105, looseness=2.00] (56.center)
			 to [bend left=105, looseness=2.00] cycle;
		\draw [style=new edge style 1,fill=gray!30!white] (59.center)
			 to [bend left=105, looseness=2.00] (58.center)
			 to [in=-15, out=0] cycle;
    	\node [style=new style 0,scale=0.6] (0) at (-0.25, 5.75) {\tiny 1};
		\node [style=new style 0,scale=0.6] (1) at (-1.25, 4.5) {\tiny 2};
		\node [style=new style 0,scale=0.6] (2) at (1, 4) {\tiny  11};
		\node [style=new style 0,scale=0.6] (3) at (-3.5, 2.5) {\tiny  3};
		\node [style=new style 0,scale=0.6] (4) at (-2.5, 0.75) {\tiny  4};
		\node [style=new style 0,scale=0.6] (5) at (1.75, 1) {\tiny  7};
		\node [style=new style 0,scale=0.6] (6) at (1, 0) {\tiny  5};
		\node [style=new style 0,scale=0.6] (7) at (2.5, 0) {\tiny 6};
		\node [style=new style 0,scale=0.6] (8) at (4.25, 3.75) {\tiny  9};
		\node [style=new style 0,scale=0.6] (9) at (3.5, 1.5) {\tiny 8};
            \node [style=new style 0,scale=0.6] (23) at (0, 2.75) {\tiny  10};
            	\node [style=none,scale=0.6] (60) at (1, 6.75) {\tiny $e_1$};
		\node [style=none,scale=0.6] (61) at (-2.75, 5.5) {\tiny  $e_2$};
		\node [style=none,scale=0.6] (62) at (-3.75, -0.5) {\tiny $e_3$};
		\node [style=none,scale=0.6] (63) at (0, -1) {\tiny $e_4$};
		\node [style=none,scale=0.6] (64) at (3, 5.75) {\tiny  $e_5$};
		\node [style=none,scale=0.6] (65) at (-0.75, 5.25) {\tiny  $W_{E_1}$};
		\node [style=none,scale=0.6] (66) at (-3, 1.5) {\tiny  $W_{E_2}$};
  \node [style=none,scale=0.6] (72) at (2, -0.75) {\tiny  $W_{E_3}$};
		\node [style=none,scale=0.6] (73) at (3.75, 2.75) {\tiny  $W_{E_4}$};
		\node [style=none,scale=0.6] (74) at (1, 4.75) {\tiny  $W_{E_5}$};
		\node [style=none,scale=0.6] (75) at (1.05, 2.5) {\tiny  $W_{E_6}$};
		\node [style=none] (76) at (0.25, -3.75) {$H$};
\end{tikzpicture}
         \caption{A hypergraph $H$ wherein units are identified within the shaded regions.}
         \label{fig:hyp-unit}
        \end{subfigure}
        \begin{subfigure}{0.5\textwidth}
        \begin{tikzpicture}[scale=0.6]
		\node [style=none] (11) at (4, 5) {};
		\node [style=none] (12) at (6, 4.25) {};
		\node [style=none] (13) at (5.5, 2.25) {};
		\node [style=none] (14) at (0.25, 0.25) {};
		\node [style=none] (15) at (2.25, -2.25) {};
		\node [style=none] (16) at (3.5, -2) {};
		\node [style=none] (17) at (-1.75, 4.25) {};
		\node [style=none] (18) at (-0.25, 6.5) {};
		\node [style=none] (19) at (1.75, 4.75) {};
		\node [style=none] (20) at (0.25, 0.5) {};
		\node [style=none] (21) at (3, -1.25) {};
		\node [style=none] (22) at (1.75, 1.75) {};
		\node [style=none] (24) at (-4.75, 2.75) {};
		\node [style=none] (25) at (-3.75, 4) {};
		\node [style=none] (26) at (-2.25, 2.25) {};
		\node [style=none] (27) at (-4, 0.75) {};
		\node [style=none] (28) at (-3, -0.25) {};
		\node [style=none] (29) at (0.75, 3.5) {};
		\node [style=none] (30) at (-0.25, 2.25) {};
		\node [style=none] (31) at (0.75, 5) {};
		\node [style=none] (32) at (5, 4) {};
		\node [style=none] (33) at (2.75, 1.75) {};
		\node [style=none] (34) at (3.75, 0.75) {};
		\node [style=none] (35) at (4.75, 1.75) {};
		\node [style=none] (36) at (-2.25, 5.75) {};
		\node [style=none] (37) at (-0.75, 6.5) {};
		\node [style=none] (38) at (0.25, 5.5) {};
		\node [style=none] (39) at (-4.25, 2.5) {};
		\node [style=none] (40) at (-2.75, 0.25) {};
		\node [style=none] (41) at (-1.5, 1) {};
		\node [style=none] (42) at (0.5, 0) {};
		\node [style=none] (43) at (1.75, 1.5) {};
		\node [style=none] (44) at (3, -1.25) {};
		\node [style=none] (45) at (3, 2.75) {};
		\node [style=none] (46) at (4.25, 4.25) {};
		\node [style=none] (47) at (3.75, 1) {};
		\node [style=none] (48) at (-4, 2) {};
		\node [style=none] (49) at (-3.25, 3.25) {};
		\node [style=none] (50) at (-2, 0.5) {};
		\node [style=none] (51) at (-1.5, 5.5) {};
		\node [style=none] (52) at (-0.25, 6.25) {};
		\node [style=none] (53) at (-1, 3.75) {};
		\node [style=none] (56) at (0.25, 3.25) {};
		\node [style=none] (57) at (0, 2.25) {};
		\node [style=none] (58) at (1.25, 5) {};
		\node [style=none] (59) at (1, 3.25) {};
		\node [style=none] (67) at (0.75, 2.25) {};
		\node [style=none] (68) at (1, -1.25) {};
		\node [style=none] (69) at (3.75, 2.75) {};
		\node [style=none] (70) at (2, 4.25) {};
		\node [style=none] (71) at (1, 2.75) {};
		\draw (14.center)
			 to [in=-150, out=45, looseness=0.75] (11.center)
			 to [bend left=45] (12.center)
			 to [bend left=45] (13.center)
			 to [in=60, out=-165, looseness=1.50] (16.center)
			 to [bend left=45] (15.center)
			 to [bend left=45] cycle;
		\draw (20.center)
			 to [in=-30, out=135] (17.center)
			 to [bend left=45] (18.center)
			 to [bend left=45] (19.center)
			 to [in=135, out=-75, looseness=0.75] (22.center)
			 to [bend left=45] (21.center)
			 to [bend left=45] cycle;
		\draw (27.center)
			 to (24.center)
			 to [bend left=45] (25.center)
			 to [in=165, out=0, looseness=1.25] (26.center)
			 to [in=180, out=0] (29.center)
			 to [bend left=45] (28.center)
			 to [bend left=45] cycle;
		\draw (33.center)
			 to [in=-45, out=165, looseness=1.25] (30.center)
			 to [bend left=45] (31.center)
			 to [bend left=45] (32.center)
			 to [in=135, out=-75, looseness=1.25] (35.center)
			 to [bend left=45] (34.center)
			 to [bend left=45] cycle;
		\draw (39.center)
			 to [in=-90, out=60] (36.center)
			 to [bend left=45] (37.center)
			 to [bend left=45] (38.center)
			 to [in=90, out=-135] (41.center)
			 to [bend left=45] (40.center)
			 to [bend left=45] cycle;
            \node [rectangle,draw] (23) at (0, 2.75) {};
            	\node [style=none] (60) at (1, 6.75) {\tiny $\tilde e_1$};
		\node [style=none] (61) at (-2.75, 5.5) {\tiny  $\tilde e_2$};
		\node [style=none] (62) at (-3.75, -0.5) {\tiny $\tilde e_3$};
		\node [style=none] (63) at (0, -1) {\tiny $\tilde e_4$};
		\node [style=none] (64) at (3, 5.75) {\tiny  $\tilde e_5$};
		\node [rectangle,draw] (65) at (-0.75, 5.25) {\tiny  $W_{E_1}$};
		\node [rectangle,draw] (66) at (-3, 1.5) {\tiny  $W_{E_2}$};
  \node [rectangle,draw] (72) at (2, -0.5) {\tiny  $W_{E_3}$};
		\node [rectangle,draw] (73) at (3.75, 2.75) {  \tiny$W_{E_4}$};
		\node [rectangle,draw] (74) at (1, 4.5) {\tiny  $W_{E_5}$};
		\node [none] (75) at (1.05, 2.5) {\tiny  $W_{E_6}$};
		\node [style=none] (76) at (0.25, -3.75) {$H/\mathcal{R}_u(H)$};
\end{tikzpicture}
         \caption{Units of $H$ become vertices in $H/\mathcal{R}_u(H)$, the unit-contraction of $H$. 
         }
         \label{fig:unit-contraction}
        \end{subfigure}
        \caption{Units and unit contraction of a hypergraph $H$  with $V(H)=\{1,2,\ldots,10,11\}$, and $E(H)=\{ e_1=\{1,2,5,6,7,10,11,\},e_2=\{1,2,3,4\}, e_3=\{3,4,10\},e_4=\{5,6,7,8,9\},e_5=\{8,9,10,11\}\}$.}
        \label{fig:unit and contraction}
    \end{figure}
  Consider the hypergraph $H$ with $V(H)=\{1,2,\ldots,10,11\}$ and $E(H)=\{e_1,e_2,e_3,e_4,e_5\}$ (see \Cref{fig:hyp-unit}), where $ e_1=\{1,2,5,6,7,10,11,\}$, $e_2=\{1,2,3,4\}, e_3=\{3,4,10\}$,
$e_4=\{5,6,7,8,9\}$, $e_5=\{8,9,10,11\}$. The units in $H$ are $ W_{E_1}=\{1,2\}$, $W_{E_2}=\{3,4\}$, $W_{E_3}=\{5,6,7\}$, $W_{E_4}=\{8,9\}$, $W_{E_5}=\{11\}$, $W_{E_6}=\{10\}$. The corresponding generating sets are $E_1=\{e_1,e_2\}$, $E_2=\{e_2,e_3\}$, $E_3=\{e_1,e_4\}$, $E_4=\{e_4,e_5\}$, $E_5=\{e_1,e_5\}$, $E_6=\{e_1,e_3,e_5\}$.
 \end{exm}
A hypergraph $H$ is called \emph{non-contractible} if each unit is a singleton set. If $H$ is non-contractible, then $H$ is isomorphic to $H/\mathcal{R}_u(H)$.
Two hypergraphs $H$ and $H'$ are called \emph{isomorphic} if there exists a bijection $f:V(H)\to V(H')$ such that $e\in E(H)$ if and only if $\{f(v):v\in e\}\in E(H')$.
\begin{prop}\label{contraction-iso}
   Let $H$ be a hypergraph, and $U\subseteq V(H)$ be such that $U$ contains exactly one vertex from each unit of $H$. The sub-hypergraph of $H$ induced by $U$, that is, $H_U$, is isomorphic to $H/\mathcal{R}_u(H)$, the unit-contraction of $H$.
\end{prop}
\begin{proof}
    Consider the function $f:V(H_U)\to V(H/\mathcal{R}_u(H))$ defined by $f(u)=W_{E_u(H)}$, the unit containing the vertex $u$. Given any $\tilde{e}\in E(H/\mathcal{R}_u(H))$, there exist $e\in E(H)$ such that $\tilde{e}=\{W_{E_v(H)}:v\in e\}$. Since the set $U$ contains exactly one vertex from each unit of $H$, the intersection $e\cap U\ne \emptyset$ for all $e\in E(H)$, and therefore, $E(H_U)=\{e\cap U:e\in E(H)\}$. For any $e\in E(H)$, the set $\tilde{e}=\{W_{E_v(H)}:v\in e\}=\{ W_{E_v(H)}:v\in e\cap U\}=\{f(v):v\in e\cap U\}$. Therefore, $e\in E(H_U)$ if and only if $\{f(v):v\in e\}\in E(H/\mathcal{R}_u(H))$.
\end{proof}
Consider the hypergraph $H$ illustrated in the \Cref{fig:hyp-unit}. Suppose that $U=\{1,3,5,8,10,11\}$. The subset $U$ contains exactly one vertex from each unit. Therefore, the sub-hypergraph of $H$ induced by $U$, that is, $H_U$, is isomorphic to $H/\mathcal{R}_u(H) $ (illustrated in the \Cref{fig:unit-contraction}).
The \Cref{thm-induced} and the \Cref{contraction-iso} lead us to the following result.
\begin{thm}\label{contraction-null}
    Let $H$ be a hypergraph. The dimension of the null space of $B_H$ is at least the dimension of the null space of $B_{H/\mathcal{R}_u(H)}$.
\end{thm}
\begin{proof}
    Let $U\subseteq V(H)$ be such that $U$ contains exactly one vertex from each unit in the hypergraph $H$. By the \Cref{contraction-iso}, the sub-hypergraph of $H_U$ is isomorphic to $H/\mathcal{R}_u(H)$. Therefore, the dimensions of the null spaces of $B_{H_U}$ and $B_{H/\mathcal{R}_u(H)}$ are the same. By the \Cref{thm-induced}, for each vector $y$ in the null space of $B_{H_U}$, its extension $y'$ belongs to the null space of $B_H$. If $y_1,\ldots,y_k$ are linearly independent vectors in the null space of $B_{H_U}$, then their extensions $y_1',\ldots,y_k'$ are also linearly independent vectors in the null space of $B_{H}$. Thus, the result follows.
\end{proof}

Suppose that $u,v\in V(H)$ are such that $E_u(H)=E_v(H)$, then $U=\{u\}$, and $V=\{v\}$ forms an equal partition of hypergraph. Thus, $\chi_{\{u\}}-\chi_{\{v\}}$ belongs to the null space of $B_H$. We denote $\chi_{\{u\}}-\chi_{\{v\}}$ as $x_{uv}$. Therefore, we have the following result.
\begin{prop}\label{prop-uv}
     Let $H$ be a hypergraph with $u,v\in V(H)$. The stars $E_u(H)=E_v(H)$ if and only if $B_Hx_{uv}=0$.
\end{prop}
\begin{proof}
 If $E_u(H)=E_v(H)$, the two columns of $B_H$ corresponding to the vertices $u$ and $v$ are identical. Therefore, $B_Hx_{uv}=0$.
Conversely, if $B_Hx_{uv}=0$, then two columns of $B_H$ corresponding to the two vertices $u$ and $v$ are identical. Therefore, $E_u(H)=E_v(H)$.
\end{proof}
Suppose that $H$ is a hypergraph, and $U\subseteq V(H)$. Consider the vector space $S_U=\{x:V(H)\to\mathbb{C}:x(v)=0 \text{~for all~} v\in V(H)\setminus U, \text{~and~}\sum\limits_{v\in U}x(v)=0\}$. If $U=\{u_0,u_1,\ldots,u_n\}$, then the collection  $\{x_{u_iu_0}:i=1,\ldots,n\}$ spans the vector space $S_U$. That is, for a unit $W_{E}=\{v_0,\ldots,v_n\}$, the vector space $S_{W_E}$ is spanned by the vectors $x_{v_1v_0},\ldots,x_{v_nv_0}$. For instance, consider the hypergraph $H$ described in the \Cref{ex-unit} (see the \Cref{fig:hyp-unit}). For the unit $W_{E_3}=\{5,6,7\}$, we have each of the two vectors $ x_{v_1v_0}$, and $x_{v_2v_0}$ belongs to the null space of $B_H$. Consequently, the two-dimensional vector space $S_{W_{E_3}}$ is a subspace of the null space of $B_H$. There does not exist any $W$ such that $W_{E_3}\subsetneq W\subseteq V(H) $ with $S_W$ is a subspace of the null space of $B_H $. That is, $W_{E_3}$ is a maximal collection of vertices with the property $S_{W_{E_3}}$ is a subspace of the null space of $B_H$. Thus, we have the following result.
\begin{thm}
    \label{null-ag}
     Let $H$ be a hypergraph, $W\subseteq V(H)$, and $|W|\ge 2$. The set  $W$ is a unit in $H$ if and only if $W$ is a maximal set such that $S_W$ is a subspace of the null space of $B_H$.
\end{thm}
\begin{proof}
    Suppose the set $W$ is a unit in $H$ and $W=W_E$. If $W=W_E=\{v_0,v_1,\ldots,v_n\}$, then by the 
\Cref{prop-uv}, $B_Hx_{v_iv_0}=0$ for all $i=1,\ldots,n$. Therefore, $S_{W_E}$ is a subspace of the null space of $B_H$. If possible, let $W$ not be a maximal subset with the property of $S_W$ as a subspace of the null space of $B_H$. Thus, there exists a $W'$ such that $W\subsetneq W'\subseteq V(H)$, with $S_{W'}$ being a subspace of the null space of $B_H$. Since $W$ is a proper subset of $W'$, there exists $u\in W'\setminus W$. Now, for any $v\in W$, the vector $x_{uv}\in S_W$. Therefore, $B_Hx_{uv}=0$. Therefore, by the \Cref{prop-uv}, $E_u(H)=E_v(H)$. This is a contradiction to the fact that $W$ is a unit in $H$ because a unit is a maximal collection of vertices with the same stars. Therefore, our assumption is wrong, and $W$ is a maximal set such that $S_W$ is a subspace of the null space of $B_H$.

Conversely, suppose that $W$ is a maximal set such that $S_W$ is a subspace of the null space of $B_H$. Therefore, by the \cref{prop-uv}, $W$ is a maximal collection of vertices with the same star. Therefore, $W$ is a unit.
\end{proof}
Since for any unit $W_E$ in $H$ with $|W_E|>2$, the dimension of $S_{W_E}$ is $|W_E|-1$, we have the following Corollary of the \Cref{null-ag}.
\begin{cor}
    Let $H$ be a hypergraph. The dimension of the null space of $B_H$ is at least $(|V(H)|-|\mathfrak{U}(H)|)$.
\end{cor}
\begin{proof}
    Since by the \Cref{null-ag}, for each unit $W_{E}$ in a hypergraph $H$ with $|W_E|>1$, we have $S_{W_E}$ is a subspace of the null space of $B_H$, and the dimension of $S_{W_E}$ is  $|W_E|-1$. Therefore, the dimension of the null space of $B_H$ is at least $ \sum\limits_{W_E\in \mathfrak{U}(H)}(|W_E|-1)=(|V(H)|-|\mathfrak{U}(H)|)$. 
\end{proof}
The above Corollary immediately indicates the following result.
\begin{cor}
    For a hypergraph $H$, the rank of $B_H$ is at most $|\mathfrak{U}(H)|$.
\end{cor}
For example, consider the hypergraph illustrated in the \Cref{fig:hyp-unit}. Since the hypergraph has $6$ units (see the  \Cref{fig:unit-contraction}), the rank of $B_H$ is at most $6$. The \Cref{null-ag} suggests that units are one of the structures in hypergraphs that induce vectors in the null space of $B_H$. 
Since we have already shown that besides units, other hypergraph structures are also responsible for the vectors in the null space of $B_H$, the rank of $B_H$ may be less than the number of units. For instance, for the above hypergraph $H$ (illustrated in \Cref{fig:hyp-unit}), the vector $y=-\frac{2}{3}\chi_{\{1,10\}}+\frac{2}{3}\chi_{\{3\}}+\frac{1}{3}\chi_{\{5\}}-\frac{1}{3}\chi_{\{8\}}+\chi_{\{11\}}$ belongs to the null space of $B_H$. This vector is not related to the unit but can be explained using the \Cref{gen-ker}.

The \Cref{null-ag} shows that each unit $W_E$ of cardinality at least $2$ leads to a vector in the null space of $B_H$. In the proof of \Cref{contraction-null}, we have shown that each vector in the null space of $B_{H/\mathcal{R}_u(H)}$ corresponds to a vector in the null space of $B_H$. These two facts motivate the following result. Given any matrix $M$, we denote the \emph{null space of $M$} as $Ker(M)$.
\begin{thm}\label{rank-null-bh}
    For a hypergraph $H$, the dimension of the null space of $B_H=$ the dimension of the null space of $B_{H/\mathcal{R}_u(H)}$ $+|V(H)|-|\mathfrak{U}(H)|$.
\end{thm}
\begin{proof}
For any vector $x\in Ker(B_H)$, we set $\hat x:\mathfrak{U}(H)\to\mathbb{C}$ as $ \hat{x}(W_E)=\sum\limits_{v\in W_E}x(v)$ for any $W_E\in \mathfrak{U}(H)$. For any $\tilde{e}\in E(H/\mathfrak{R}_u(H))$, there exists $e\in E(H)$ such that $\tilde{e}=\{W_E:W_E\subseteq e\}$. Therefore, 
\begin{align*}
    (B_{H/\mathcal{R}_u(H)}\hat{x})(\tilde{e})&=\sum\limits_{W_E\in\mathfrak{U}(H)}b_{\tilde{e}W_E}\hat{x}(W_E)\\
    &=\sum\limits_{W_E\in\mathfrak{U}(H)}b_{\tilde{e}W_E}\sum\limits_{v\in W_E}x(v)\\
     &=\sum\limits_{v\in V(H)}b_{ev}x(v)=(B_Hx)(e).
\end{align*}
Since $x\in Ker(B_H)$, it follows that $ (B_{H/\mathcal{R}_u(H)}\hat{x})(\tilde{e})=(B_Hx)(e)=0$ for all $\tilde{e}\in E(H/\mathfrak{R}_u(H))$. Thus, $B_{H/\mathcal{R}_u(H)}\hat{x}=0$, leading to a mapping $f:Ker(B_H)\to Ker(B_{H/\mathcal{R}_u(H)})$ defined by $f(x)=\hat{x}$ for all $x\in Ker(B_H)$.

Now, we show that $f$ is surjective. Let $y\in Ker(B_{H/\mathcal{R}_u(H)})$. Consider $U\subseteq V(H)$ such that $U$ contains exactly one element from each unit. Note that for each $u\in U$, we have $W_{E_u(H)}$ is the unit containing $u$. Define $y':V(H)\to\mathbb{C}$ as $y'(u)=y(W_{E_u(H)})$ if $u\in U$ and otherwise, $y'(u)=0 $ if $u\in V(H)\setminus U$. Since \(y'(u) \neq 0\) implies \(u \in U\), each unit can contain at most one element where \(y'\) is non-zero. Therefore, $f(y') =\hat{y'}=y$. Consequently, $f$ is surjective.
Since $f$ is a surjective linear map, by the rank-nullity theorem \cite[Chap.3, Sec.3.1]{LA}, dimension of $Ker(B_H)=$ dimension of the null space of $f+$ dimension of $Ker(B_{H/\mathcal{R}_u(H)}) $. Since $x\in \bigoplus\limits_{W_E\in\mathfrak{U}(H), |W_E| > 1}S_{W_E}$ if and only if $\hat{x}(W_E)=\sum\limits_{v\in W_E}x(v)=0$ for any $W_E\in \mathfrak{U}(H)$, the subspace $\bigoplus\limits_{W_E\in\mathfrak{U}(H), |W_E| > 1}S_{W_E} $ is the null space of $f$. Since the dimension of $\bigoplus\limits_{W_E\in\mathfrak{U}(H), |W_E| > 1}S_{W_E} $ is $\sum\limits_{W_E\in\mathfrak{U}(H)}(|W_E|-1))=|V(H)|-|\mathfrak{U}(H)|$, the result follows.
\end{proof}
If we consider the hypergraph $H$, illustrated in the  \Cref{fig:hyp-unit}, the sum $\sum\limits_{W_E\in\mathfrak{U}(H)}(|W_E|-1))=|V(H)|-6=5$. The dimension of the null space of $B_{H/\mathcal{R}_u(H)}$ is $ 1$. Therefore, as the \Cref{rank-null-bh} suggests, the dimension of the null space of $B_H$ is $1+5=6$. It is intriguing to note that the rank of $B_H$ is $5$, which is the same as the rank of $B_{H/\mathcal{R}_u(H)}$. Generally, this fact can be proved for any hypergraph as a direct consequence of the \Cref{rank-null-bh}.
\begin{cor}\label{quotient-rank}
    For any hypergraph $H$, the rank of $B_H$ and the rank of $B_{H/\mathcal{R}_u(H)}$ are the same.
\end{cor}
\begin{proof}
    By the rank-nullity theorem, the rank of $B_H$=$|V(H)|-$ dimension of the null space of $B_H$. Therefore, by the \Cref{rank-null-bh}, the rank of $B_H=|\mathfrak{U}(H)|-$ the dimension of the null space of $B_{H/\mathcal{R}_u(H)}=$ the rank of $B_{H/\mathcal{R}_u(H)}$.
\end{proof}
We have earlier established that for any hypergraph \(H\), the vertex-edge incidence matrix \(I_H\) is the transpose of \(B_H\), that is, \(I_H = B_H^T\). In our prior findings, we explored various relationships between the structure of the hypergraph and the null space of the matrix \(B_H\). In this section, we will examine similar properties related to the null space of \(I_H\). As noted in \Cref{thm-equal-partition}, any equal partition of hyperedges corresponds to a vector in the null space of \(B_H\). The following definition introduces an equal partition of vertices, which will yield a vector in the null space of \(I_H\).
\begin{df}[Equal partition of vertices]
Let $H$ be a hypergraph. A pair $E,F\subset E(H)$ with $E\cap F=\emptyset$ is called an \emph{equal partition} of vertices if $|E_v(H)\cap E|=|E_v(H)\cap F|$ for all the vertices $v\in V(H)$.
\end{df} For instance, for any natural number $n$, consider the cycle graph $C_{2n}$. Suppose that $\alpha:E(H)\to \{-1,1\}$ is such that $\alpha(e_i)=(-1)^i$ for all $i=1,\ldots,n$. The pair of sets $\alpha^{-1}(-1)$, and $\alpha^{-1}(1)$ form an equal partition of vertices.
\begin{thm} \label{thm-equal partition of vertices}Let $H$ be a hypergraph.
    Two disjoint collections of hyperedges $E$, and $F$ are an equal partition of vertices if and only if the vector $\alpha=\chi_{E}-\chi_{F}$ belongs to the null space of $I_H$. 
\end{thm}
\begin{proof}
   Suppose that \( E \) and \( F \) form an equal partition of the vertices. Consequently, \( |E_v(H) \cap E| = |E_v(H) \cap F| \) for every vertex \( v \in V(H) \). Thus, \( (I_H\alpha)(v) = |E_v(H) \cap E| - |E_v(H) \cap F| = 0 \) for all \( v \in V(H) \), and $\alpha$ belongs to the null space of $I_H$.

   Conversely, if $\alpha=\chi_{E}-\chi_{F}$ belongs to the null space of $I_H$, then \(  |E_v(H) \cap E| - |E_v(H) \cap F| =(I_H\alpha)(v) = 0 \). 
   \end{proof}
\begin{exm}
  Let $H$ be a hypergraph with $V(H)=\{1,2,3,4,5\}$, and $E(H)=\{e_1=\{1,2,3\},e_2=\{1,3,4\},e_3=\{1,4,5\},e_4=\{1,5,2\}\}$. The pair of disjoint collections of hyperedges $E=\{e_1,e_3\}$, and $E_2=\{e_2,e_4\}$ forms an equal partition of vertices in $H$. The vector $\alpha=\chi_E-\chi_F$ belongs to the null space of $I_H$.

\end{exm}
Consider the graph $G$ with $V(G)=\{1,2,3,4\}$, and $E(G)=\{e_1=\{1,2\},e_2=\{3,4\},e_3=\{1,3\},e_4=\{1,4\},e_5=\{2,3\},e_6=\{2,4\}\}$. For the pair of disjoint subsets of hyperedges $E=\{e_1,e_2\}$, and $F=E(H)\setminus E$.  For any vertex $v\in V(G)$, it holds that $|E_v(G)\cap E|:|E_v(G)\cap F|=1:2$, and $2\chi_E-\chi_F$ is a vector in the null space of $I_H$. This instance motivates the possibility of a result similar to the \Cref{ratio-r} is also true for  $I_H$.
\begin{thm}\label{ratio-r-edge}
    Let $H$ be a hypergraph. For two subsets $E,F\subset E(H)$ with $E\cap F=\emptyset$, the ratio $|E_v(H)\cap E|:|E_v(H)\cap F|=r$ if and only if the vector $\alpha=\chi_E-r\chi_F$ belongs to the null space of $I_H$.
\end{thm}
\begin{proof} If $|E_v(H)\cap E|:|E_v(H)\cap F|=r$, then
    for any $v\in V(H)$, it holds that $(I_H\alpha)(v)=\sum\limits_{e\in E_v(H)}\alpha(e)=|E_v(H)\cap E|-r|E_v(H)\cap F|=0$.

    Conversely, if the vector $\alpha=\chi_E-r\chi_F$ belongs to the null space of $I_H$, then $|E_v(H)\cap E|-r|E_v(H)\cap F|=\sum\limits_{e\in E_v(H)}\alpha(e)=(I_H\alpha)(v)=0$ for all $v\in V(H)$. Therefore, the ratio $|E_v(H)\cap E|:|E_v(H)\cap F|=r$.
\end{proof}
Though the \Cref{ratio-r-edge} is proven here independently, the theorem can be proven as a direct consequence of the \Cref{ratio-r}. Given a hypergraph $H$, the \emph{dual} of $H$ is the hypergraph $H^*$ such that $V(H^*)=E(H)$, and $E(H^*)=\{E_v(H):v\in V(H)\}$. For any hypergraph $H$, two subsets $E,F\subset E(H)$ with $E\cap F=\emptyset$, the ratio $|E_v(H)\cap E|:|E_v(H)\cap F|=r$ becomes two disjoint subsets $E$ and $F$ of the vertex set $V(H^*)$. Thus, the \Cref{ratio-r-edge} follows from the \Cref{ratio-r}. Similarly, results similar to the \Cref{uvw} and the \Cref{gen-ker} can be deduced for $I_H$.
\section{Incidence matrix and eigenvalues of other Hypergraph matrices}
Each column of the edge-vertex incidence matrix $B_H$ corresponds to a vertex of the hypergraph $H$. In this section, we show that this fact leads to some relation between the incidence matrix $B_H$ and some other matrices associated with the hypergraph $H$. One such matrix is the adjacency matrix $A_{(1,H)}=[a_{uv}]_{u,v\in V(H)}$ described in \cite{Bretto-hypergraph}, which is defined as $ a_{uv}=|E_u(H)\cap E_v(H)|$ for two distinct $u,v\in V(H)$, and the diagonal entries are $0$. Each unit $W_E$ in hypergraph $H$ with $|W_E|>1$ leads to an eigenvalue of $A_{(1,H)}$ (see \cite[Section-3]{unit}). Now, we show that this eigenvalue is related to the edge-vertex incidence matrix $B_H$. Before going into this result, recall that each column of $B_H$ is indexed by a vertex $v\in V(H)$; we denote the column as $s_v$. For two vectors $x:E(H)\to \mathbb{C}$ and $y:E(H)\to\mathbb{C}$, the usual inner product $\langle\cdot,\cdot\rangle$ is defined as $\langle x,y\rangle=\sum\limits_{e\in E(H)}x(v)\overline{y(v)}$. It is intriguing to note that for two vertices $u$ and $v$, the inner product $\langle s_u,s_v\rangle=\sum\limits_{e\in E(H)}b_{eu}b_{ev}=|E_u(H)\cap E_v(H)| $. This fact leads to the following theorem.
\begin{thm}\label{adj-rdr}
    Let $H$ be a hypergraph. For each unit $W_E$ in $H$ with $|W_E|>1$, the adjacency matrix $A_{(1,H)}$ has an eigenvalue $-\langle s_u,s_v\rangle$ of multiplicity $|W_E|-1$, where $u,v(\ne u)\in W_E$.
\end{thm}
\begin{proof}
Since $|W_E|>1$, there exists $u,v(\ne u)\in W_E$. Consider the vector $x_{uv}=\chi_{\{u\}}-\chi_{\{v\}}$. Since $u,v\in W_E$, it holds that $a_{wu}=E\cap E_w(H)=a_{wv}$ for any $w\in V(H)\setminus\{u,v\}$. Consequently, $a_{uv}=|E_u(H)\cap E_v(H)|=a_{vu}$, leads us to $A_{(1,H)}x_{uv}=-|E_u(H)\cap E_v(H)|x_{uv} =-\langle s_u,s_v\rangle x_{uv}$. 

Again, if $W_E= \{v_0,v_1,\ldots,v_k\}$ then  $\langle s_{v_0},s_{v_i} \rangle=|E_{v_0}\cap E_{v_i}|=|E|$, and by the above argument $A_{(1,H)}x_{v_0v_i}=-\langle s_{v_0},s_{v_i} \rangle x_{v_0v_i}$ for all $i=1,\ldots,k$. Since $x_{v_0v_1},x_{v_0v_2},\ldots,x_{v_0v_k}$ are linearly independent, the multiplicity of the eigenvalue is $|W_E|-1$. 
\end{proof}
\begin{exm}\label{exm-adj-rdr}\rm
    Consider the hypergraph \(H\) shown in Figure \Cref{fig:hyp-unit}. The matrix representation \(A_{(1,H)}\) is given by:
\[
A_{(1,H)} = \left[
\begin{smallmatrix}
  0 & 2 & 1 & 1 & 1 & 1 & 1 & 0 & 0 & 1 & 1\\ 
  2 & 0 & 1 & 1 & 1 & 1 & 1 & 0 & 0 & 1 & 1\\ 
  1 & 1 & 0 & 2 & 0 & 0 & 0 & 0 & 0 & 1 & 0\\ 
  1 & 1 & 2 & 0 & 0 & 0 & 0 & 0 & 0 & 1 & 0\\ 
  1 & 1 & 0 & 0 & 0 & 2 & 2 & 1 & 1 & 1 & 1\\ 
  1 & 1 & 0 & 0 & 2 & 0 & 2 & 1 & 1 & 1 & 1\\ 
  1 & 1 & 0 & 0 & 2 & 2 & 0 & 1 & 1 & 1 & 1\\ 
  0 & 0 & 0 & 0 & 1 & 1 & 1 & 0 & 2 & 1 & 1\\ 
  0 & 0 & 0 & 0 & 1 & 1 & 1 & 2 & 0 & 1 & 1\\ 
  1 & 1 & 1 & 1 & 1 & 1 & 1 & 1 & 1 & 0 & 2\\ 
  1 & 1 & 0 & 0 & 1 & 1 & 1 & 1 & 1 & 2 & 0 
\end{smallmatrix}
\right].
\]

Each row \(i\) and column \(j\) of this matrix correspond to vertices \(i\) and \(j\) for \(i, j = 1, 2, \ldots, 11\). The edge-vertex incidence matrix \(B_H\) is:

\[
B_H = 
\begin{smallmatrix}
   \hline|\phantom{e_1}||&1&2&3&4&5&6&7&8&9&10&11|\\\hline\hline
   |e_1||&1&1&0&0&1&1&1&0&0&1&1\phantom{1}|\\
   |e_2||&1&1&1&1&0&0&0&0&0&0&0\phantom{1}|\\
    |e_3||&0&0&1&1&0&0&0&0&0&1&0\phantom{1}|\\
    |e_4||&0&0&0&0&1&1&1&1&1&0&0\phantom{1}|\\
     |e_5||&0&0&0&0&0&0&0&1&1&1&1\phantom{1}|\\\hline
\end{smallmatrix}.
\]

The units \(W_{E_1} = \{1, 2\}, W_{E_2} = \{3, 4\}, W_{E_3} = \{5, 6, 7\}, W_{E_4} = \{8, 9\}\) correspond to the eigenvalues \(-\langle s_1, s_2 \rangle = -2\), \(-\langle s_3, s_4 \rangle = -2\), \(-\langle s_5, s_6 \rangle = -2\), and \(-\langle s_8, s_9 \rangle = -2\), respectively. The unit \(W_{E_3}\) contributes a multiplicity of at least 2, and each of the other units contributes a multiplicity of at least 1. Thus, by \Cref{adj-rdr}, \(-2\) is an eigenvalue of \(A_{(1,H)}\) with multiplicity at least $5$.
\end{exm}
It is important to note that, to compute this eigenvalue, we have not relied on the entries or any other specific details of the matrix \(A_{(1,H)}\); instead, we have used only the columns of the matrix \(B_H\). This observation naturally raises the question of whether the matrix \(B_H\) itself contains all the necessary information to determine these eigenvalues of \(A_{(1,H)}\). In the next two results, we show that the matrix \(B_H\), along with the inner product, indeed encapsulates this information.

Another variation of hypergraph adjacency $A_{(2,H)}=[a_{uv}]_{u,v\in V(H)}$ is described in \cite{hg-mat}. Let $H$ be a hypergraph with $|e|>1$ for all $e\in E(H)$. For two distinct vertices $u,v\in V(H)$, the $(u,v)$-th entry of the matrix $A_{(2,H)}$ is $a_{uv}=\sum\limits_{e\in E_u(H)\cap E_v(H)}\frac{1}{|e|-1}$, and all the diagonal entries of the matrix are $ 0$. For this matrix also, we can conclude a result similar to the \Cref{adj-rdr}. We just need to change the inner product. For two vectors $x:E(H)\to \mathbb{C}$ and $y:E(H)\to\mathbb{C}$, we define an inner product $(x,y)=\sum\limits_{e\in E(H)}\frac{1}{|e|-1}x(e)y(e)$. This inner product is well-defined for all hypergraphs with $|e|>1$ for all $e\in E(H)$.
\begin{thm}\label{adj-ban}
    Let $H$ be a hypergraph with $|e|>1$ for all hyperedges $e\in E(H)$. For each unit $W_E$ in $H$ with $|W_E|>1$, the adjacency matrix $A_{(2,H)}$ has an eigenvalue $-( s_u,s_v)$ of multiplicity $|W_E|-1$, where $u,v(\ne u)\in W_E$.
\end{thm}
\begin{proof}
    The inner product $( s_u,s_v)=\sum\limits_{e\in E(H)}\frac{1}{|e|-1}b_{eu}b_{ev}=\sum\limits_{e\in E_u(H)\cap E_v(H)}\frac{1}{|e|-1}=a_{uv}$. Using this fact and proceeding exactly similar to the proof of the\Cref{adj-rdr}, the theorem follows.
\end{proof}
\begin{exm}\rm \label{ex-adj-ban}
   Let us start with the same hypergraph $H$ (illustrated in the \Cref{fig:hyp-unit}) and the incidence matrix $B_H$ that we have considered in the \Cref{exm-adj-rdr}. 
The matrix $A_{(2,H)}=\frac{1}{12}\left[
\begin{smallmatrix}
     0 & 6 & 4 & 4 & 2 & 2 & 2 & 0 & 0 & 2 & 2\\ 6 & 0 & 4 & 4 & 2 & 2 & 2 & 0 & 0 & 2 & 2\\ 4 & 4 & 0 & 10 & 0 & 0 & 0 & 0 & 0 & 6 & 0\\ 4 & 4 & 10 & 0 & 0 & 0 & 0 & 0 & 0 & 6 & 0\\ 2 & 2 & 0 & 0 & 0 & 5 & 5 & 3 & 3 & 2 & 2\\ 2 & 2 & 0 & 0 & 5 & 0 & 5 & 3 & 3 & 2 & 2\\ 2 & 2 & 0 & 0 & 5 & 5 & 0 & 3 & 3 & 2 & 2\\ 0 & 0 & 0 & 0 & 3 & 3 & 3 & 0 & 7 & 4 & 4\\ 0 & 0 & 0 & 0 & 3 & 3 & 3 & 7 & 0 & 4 & 4\\ 2 & 2 & 6 & 6 & 2 & 2 & 2 & 4 & 4 & 0 & 6\\ 2 & 2 & 0 & 0 & 2 & 2 & 2 & 4 & 4 & 6 & 0 
\end{smallmatrix}
\right]$.  As the \Cref{adj-ban} suggested, the units \(W_{E_1} = \{1, 2\}, W_{E_2} = \{3, 4\}, W_{E_3} = \{5, 6, 7\}, W_{E_4} = \{8, 9\}\) correspond to the eigenvalue \(-( s_1, s_2) = -\frac{1}{2}\) with multiplicity $1$, \(-( s_3, s_4 ) = -\frac{5}{6}\) with multiplicity $1$, \(-( s_5, s_6 ) = -\frac{5}{12}\) with multiplicity $2$, and \(-( s_8, s_9)= -\frac{7}{12}\) with multiplicity $1$, respectively.
\end{exm}
 The \Cref{adj-rdr} and the \Cref{adj-ban} can be generalised further using a positive valued hyperedge function. Let $w:E(H)\to(0,\infty)$ be a positive valued function. Consider the inner product $(\cdot,\cdot)_w$ such that for two vectors $x:E(H)\to \mathbb{C}$ and $y:E(H)\to\mathbb{C}$, it holds that $(x,y)_w=\sum\limits_{e\in E(H)}w(e)x(e)y(e)$. Now, we can define an adjacency matrix $A_{(w,H)}=[a_{uv}]_{u,v\in V(H)}$ such that for two distinct vertices $u$ and $v$, the entry $a_{uv}=\sum\limits_{e\in E_u(H)\cap E_v(H)}w(e)$, and all the diagonal entries of the matrix are $0$. This matrix is a generalisation of both $A_{(1,H)}$ and $A_{(2,H)}$. For this matrix, we can conclude the following:
\begin{thm}\label{adj-w}
    Let $H$ be a hypergraph. For each unit $W_E$ in $H$ with $|W_E|>1$, the adjacency matrix $A_{(w,H)}$ has an eigenvalue $-( s_u,s_v)_w$ of multiplicity $|W_E|-1$, where $u,v(\ne u)\in W_E$.
\end{thm}
\begin{proof}
    The proof is exactly similar to the proof of the \Cref{adj-rdr} and follows from the fact that for two distinct $u,v\in V(H)$, the inner product $ (s_u,s_v)_w=\sum\limits_{e\in E_u(H)\cap E_v(H)}w(e)$.
\end{proof}
This type of result holds not only for units but also for other structural symmetries of hypergraphs. For two equivalence relations $\mathfrak{R}_1$ and $\mathfrak{R}_2$, we say $\mathfrak{R}_2$ is finer than $\mathfrak{R}_1$ if $(u,v)\in\mathfrak{R}_2$ implies that $(u,v)\in\mathfrak{R}_1$. Given any matrix $M=[m_{uv}]_{u,v\in V(H)}$ associated with hypergraph $H$, consider the equivalence relation $\mathfrak{R}_M$ defined as $\mathfrak{R}_M=\{(u,v)\in V(H)\times V(H): m_{uu}=m_{vv},m_{uv}=m_{vu}\text{~and~}m_{uw}=m_{vw}, m_{wu}=m_{wv}\text{~for all~}w\in V(H)\setminus \{u,v\}\} $. Given a matrix $A$ associated with a hypergraph, if any equivalence relation $\mathfrak{R}$ on the vertex set $V(H)$ is finer than $\mathfrak{R}_A$, then each $\mathfrak{R}$-equivalence class $W$ with $|W|>1$ corresponds to an eigenvalue of $A$ with multiplicity $|W|-1$ (\cite{unit}). For a distinct pair of vertices $u,v\in W$, the vector $x_{uv}$ is an eigenvector of the eigenvalue. Now, we prove that for the adjacency matrix $A_{(w,H)}$, these eigenvalues are related to the inner product of columns of $B_H$.
\begin{thm}\label{adj-weight-finer}
    Let $H$ be a hypergraph. If $\mathfrak{R}$ is an equivalence relation on $V(H)$ such that $\mathfrak{R}$ is finer than the equivalence relation $\mathfrak{R}_{A_{(w,H)}}$, then for each $\mathfrak{R}$-equivalence class $W$, the adjacency matrix $A_{(w,H)}$ has an eigenvalue $-( s_u,s_v)_w$ of multiplicity $|W|-1$, where $u,v(\ne u)\in W$.
\end{thm}
\begin{proof}
    Since $(u,v)\in \mathfrak{R}$, and $\mathfrak{R}$ is finer than $\mathfrak{R}_{A_{(w,H)}}$, it holds that $(u,v)\in \mathfrak{R}_{A_{(w,H)}}$. Therefore, $a_{uu}=a_{vv}=0,a_{uv}=a_{vu}$, and $a_{uu'}=a_{vu'}, a_{u'u}=a_{u'v}$ for all $u'\in V(H)\setminus \{u,v\}\}$. Thus, for the vector $x_{uv}=\chi_{\{u\}}-\chi_{\{v\}}$, we have for any $u'\in V(H)$, $(A_{(w,H)}x_{uv})(u')=a_{u'u}-a_{u'v}$. For all $u'\in V(H)\setminus \{u,v\}\}$, since $a_{uu'}=a_{vu'}, a_{u'u}=a_{u'v}$ we have $(A_{(w,H)}x_{uv})(u')=0$. Since the diagonal entries of $A_{(w,H)}$ are $0$, it holds that $(A_{(w,H)}x_{uv})(u)=-a_{uv}=-a_{vu}= -(A_{(w,H)}x_{uv})(v)$. Therefore, $A_{(w,H)}x_{uv}=-a_{uv}x_{uv}=-( s_u,s_v)_wx_{uv}$. If $W=\{v_0.v_1,\ldots,v_k\}$, then $( s_{v_0},s_{v_1})=( s_{v_0},s_{v_2})=\ldots=( s_{v_0},s_{v_k})$, and $A_{(w,H)}x_{v_0v_i}=-( s_{v_0},s_{v_i})_wx_{v_0v_i}$ for all $i=1,\ldots,k$. Since $x_{v_0v_1},\ldots,x_{v_0v_k}$ are linearly independent, the multiplicity of this eigenvalue is $|W|-1$. 
\end{proof}
For the weight function $w:E(H)\to (0,\infty)$ defined by $w(e)=1$ for all $e\in E(H)$, the adjacency matrix $A_{(w,H)}=A_{(1,H)}$. For a hypergraph $H$ with, if $|e|>1$ for all $e\in E(H)$, if we set $w:E(H)\to (0,\infty)$ as $w(e)=\frac{1}{|e|-1}$, then $A_{(w,H)}=A_{(2,H)}$. Therefore, the \Cref{adj-weight-finer} holds for both the matrices $A_{(1,H)}$ and $A_{(2,H)}$. That is, if $\mathfrak{R}$ is finer than $\mathfrak{R}_{A_{(1,H)}}$, then any $\mathfrak{R}$-equivalence class $W$ with $|W|>1$ corresponds to the eigenvalue $\langle s_u,s_v\rangle$ for $u,v(\ne u)\in W$. Similarly for the matrix ${A_{(2,H)}}$, the eigenvalue is $ (s_u,s_v)$. Since the equivalence class $\mathfrak{R}_u$ is finer than $\mathfrak{R}_{A_{(w,H)}}$, \Cref{adj-rdr}, \Cref{adj-ban}, and \Cref{adj-w} can be proved as a Corollary of the \Cref{adj-weight-finer}. In the next example, we show that besides units, the \Cref{adj-weight-finer} can be applied for other hypergraph symmetries as well.
\begin{exm}\rm\label{ex-adj-w}
    Consider the hypergraph $H$ with the vertex set $V(H)=\{1,2,3,4\}$, and the hyperedge set $E(H)=\{e_1=\{1,2,3\},e_2=\{1,3,4\},e_3=\{1,4,2\}\}$. The edge-vertex incidence matrix $B_H=\begin{smallmatrix}\hline
      | \phantom{e_1} |&1&2&3&4|\\\hline\hline
      |  e_1 |&1&1&1&0|\\\hline | e_2 |&1&0&1&1|\\\hline |e_3 |&1&1&0&1|\\\hline
    \end{smallmatrix}$. Consider the equivalence relation $\mathfrak{R}=\{(1,1),(2,2),(3,3),(4,4),(2,3),(3,2),(3,4),(4,3),(2,4),(4,2)\}$. The two $\mathfrak{R}$-equivalence classes are $\{1\},\{2,3,4\}$. The adjacency matrix $A_{(1,H)}=\left[\begin{smallmatrix}
   0&2&2&2\\2&0&1&1\\2&1&0&1\\2&1&1&0
    \end{smallmatrix}\right]$.
    This matrix is a symmetric matrix with all its diagonal entries being $0$, and for any two vertices $u,v\in\{2,3,4\}$, $a_{uw}=a_{vw}$ for any $w\notin\{u,v\}$; therefore, $\mathfrak{R}$ is finer than $\mathfrak{R}_{A_{(1, H)}}$. Consequently, by the the \Cref{adj-weight-finer}, $-\langle s_2,s_3\rangle=-1$ is an eigenvalue of $\mathfrak{R}_{A_{(1,H)}}$. Moreover, $x_{12}$ and $x_{13}$ are two linearly independent eigenvectors of $-1$. Therefore, the multiplicity of this eigenvalue is at least $2$. For any $w:E(H)\to(0,\infty)$ with $w(e_1)=w(e_2)=w(e_3)$, the matrix $A_{(w,H)}=\left[
    \begin{smallmatrix}
        0&w(e_1)+w(e_3)&w(e_1)+w(e_2)&w(e_2)+w(e_3)\\w(e_1)+w(e_3)&0&w(e_1)&w(e_3)\\w(e_1)+w(e_2)&w(e_1)&0&w(e_2)\\w(e_2)+w(e_3)&w(e_3)&w(e_2)&0
    \end{smallmatrix}
    \right]$.
    By the \Cref{adj-w}, $-(s_2,s_3)_w=-w(e_1)$ is an eigenvalue of $A_{(w,H)}$ with multiplicity $2$.
\end{exm}
Since in the above example, $|e_1|=|e_2|=|e_3|$, for any $w: E(H)\to (0,\infty)$ that depends only on the cardinality of hyperedges, the condition $w(e_1)=w(e_2)=w(e_3)$ holds. For instance, in a hypergraph with $|e|>1$ for all $e\in E(H)$, if we set $w:E(H)\to(0,\infty)$ as $w(e)=\frac{1}{|e|-1}$, then this $w$ depends only on the cardinality of hyperedges. Thus, for the hypergraph $H$ considered in the above example, $-(s_2,s_3)_w=-w(e_1)=-\frac{1}{2}$ is an eigenvalue of $A_{(2,H)}$.
\section{Conclusion}In the context of graphs, for an even cycle graph $C_{2n}$, the matrix $B_{C_{2n}}$ can never be of full rank. Here we present a hypergraph analogue of these even cycles in \Cref{x-w-prop} and \Cref{gcd}. We observed like a cycle $C_{2n}$, for a $k$-uniform cycle of length $n$, $C^k_n$, the edge-vertex incidence matrix $B_{C^k_n}$ can never be of full rank if $g.c.d(k,n)=r>1$. For any $r$-th root of unity $\omega(\ne 1)$, vector $x_\omega$ always belongs to the null space of $B_{C^k_n}$. The study of rank and incidence matrices also leads to a hypergraph analogue of the bipartite graph: an equal partition of hyperedges. Since even cycles and bipartite graphs exhibit many desirable properties, it would be intriguing to study how far these hypergraph analogues exhibit similar properties. Some hypergraph structures that would decrease the rank of the vertex-edge incidence matrix are presented here. Thus, by \Cref{thm-induced}, given that a hypergraph $H$ has an induced sub-hypergraph $H'$ such that $H'$ contains any of these substructures, then $B_H$ can never have the full rank. Considering the connection of the incidence matrix and some variations of the adjacency matrix observed here, the interrelation between the incidence matrix and the eigenvalues of the other matrices associated with hypergraphs would be an intriguing direction to explore.
\section*{Acknowledgement} The author, Samiron Parui, expresses gratitude to the National Institute of Science Education and Research for their financial support through the Department of Atomic Energy plan project RIA 4001 (NISER). 

\begin{thebibliography}{10}

\bibitem{akbari-incidence}
{\sc S.~Akbari, N.~Ghareghani, G.~B. Khosrovshahi, and H.~Maimani}, {\em The kernels of the incidence matrices of graphs revisited}, Linear Algebra Appl., 414 (2006), pp.~617--625.

\bibitem{atik-equitable}
{\sc F.~Atik}, {\em On equitable partition of matrices and its applications}, Linear Multilinear Algebra, 68 (2020), pp.~2143--2156.

\bibitem{hg-mat}
{\sc A.~Banerjee}, {\em On the spectrum of hypergraphs}, Linear Algebra Appl., 614 (2021), pp.~82--110.

\bibitem{unit}
{\sc A.~{Banerjee} and S.~{Parui}}, {\em {On some building blocks of hypergraphs}}, arXiv e-prints,  (2022), p.~arXiv:2206.09764.

\bibitem{my1st}
{\sc A.~Banerjee and S.~Parui}, {\em On some general operators of hypergraphs}, Linear Algebra Appl., 667 (2023), pp.~97--132.

\bibitem{equitable-me-2024}
\leavevmode\vrule height 2pt depth -1.6pt width 23pt, {\em Spectral decomposition of hypergraph automorphism compatible matrices}, arXiv preprint arXiv:2405.01364,  (2024).

\bibitem{invariantsubspaces-me}
{\sc A.~Banerjee and S.~Parui}, {\em Symmetries of hypergraphs and some invariant subspaces of matrices associated with hypergraphs}, 2024.

\bibitem{Berge-hypergraph}
{\sc C.~Berge}, {\em Hypergraphs}, vol.~45 of North-Holland Mathematical Library, North-Holland Publishing Co., Amsterdam, 1989.
\newblock Combinatorics of finite sets, Translated from the French.

\bibitem{Bretto-hypergraph}
{\sc A.~Bretto}, {\em Hypergraph theory}, Mathematical Engineering, Springer, Cham, 2013.
\newblock An introduction.

\bibitem{Brualdi-comb-mat}
{\sc R.~A. Brualdi and H.~J. Ryser}, {\em Combinatorial matrix theory}, vol.~39 of Encyclopedia of Mathematics and its Applications, Cambridge University Press, Cambridge, 1991.

\bibitem{Trevisan_signless}
{\sc K.~Cardoso and V.~Trevisan}, {\em The signless {L}aplacian matrix of hypergraphs}, Spec. Matrices, 10 (2022), pp.~327--342.

\bibitem{Cvetkovic-1980}
{\sc D.~s.~M. Cvetkovi\'c, M.~Doob, and H.~Sachs}, {\em Spectra of graphs}, vol.~87 of Pure and Applied Mathematics, Academic Press, Inc. [Harcourt Brace Jovanovich, Publishers], New York-London, 1980.
\newblock Theory and application.

\bibitem{Benjamin-webb-gen-equitable-2019}
{\sc A.~Francis, D.~Smith, and B.~Webb}, {\em General equitable decompositions for graphs with symmetries}, Linear Algebra Appl., 577 (2019), pp.~287--316.

\bibitem{agt-godsil-royle}
{\sc C.~Godsil and G.~Royle}, {\em Algebraic graph theory}, vol.~207 of Graduate Texts in Mathematics, Springer-Verlag, New York, 2001.

\bibitem{haemer-singular-incidence}
{\sc W.~H. Haemers}, {\em Conditions for singular incidence matrices}, J. Algebraic Combin., 21 (2005), pp.~179--183.

\bibitem{LA}
{\sc K.~Hoffman and R.~Kunze}, {\em Linear algebra}, Prentice-Hall, Inc., Englewood Cliffs, NJ, second~ed., 1971.

\bibitem{Swarup-panda-2022-hypergraph}
{\sc S.~S. Saha, K.~Sharma, and S.~K. Panda}, {\em On the {L}aplacian spectrum of {$k$}-uniform hypergraphs}, Linear Algebra Appl., 655 (2022), pp.~1--27.

\bibitem{Sarkar-banerjee-2020}
{\sc A.~Sarkar and A.~Banerjee}, {\em Joins of hypergraphs and their spectra}, Linear Algebra Appl., 603 (2020), pp.~101--129.

\bibitem{nufflen-incidence-1}
{\sc C.~van Nuffelen}, {\em On the rank of the incidence matrix of a graph}, Cah. Cent. {\'E}tud. Rech. Op{\'e}r., 15 (1973), pp.~363--365.

\end{thebibliography}

 \end{document}